\documentclass[11pt,reqno]{amsart}
 
\usepackage[T1]{fontenc}
\usepackage[utf8]{inputenc}
\usepackage{lmodern}
\usepackage[a4paper,margin=1.08in]{geometry}
\usepackage{amsmath,amssymb,amsfonts,amsthm,mathtools,mathrsfs}
\usepackage{microtype}
\usepackage{enumitem}
\usepackage{url}
\usepackage[colorlinks=true,linkcolor=blue,citecolor=blue,urlcolor=blue,filecolor=blue]{hyperref}
\usepackage{color}

\allowdisplaybreaks
\numberwithin{equation}{section}

\newtheorem{theorem}{Theorem}[section]

\newtheorem{lemma}[theorem]{Lemma}
\newtheorem{corollary}[theorem]{Corollary}

\theoremstyle{definition}

\theoremstyle{remark}
\newtheorem{remark}[theorem]{Remark}
\theoremstyle{remark}

\hypersetup{
  colorlinks=true,
  linkcolor=blue,
  citecolor=blue,
  urlcolor=blue,
  filecolor=blue,
  pdftitle={Weak type (1,1) estimates for noncommutative rough   singular integrals}
}

\newcommand{\M}{\mathcal M}
\newcommand{\A}{\mathcal A}
\newcommand{\Q}{\mathcal Q}
\newcommand{\R}{\mathbb R}
\newcommand{\Z}{\mathbb Z}
\newcommand{\C}{\mathbb C}
\newcommand{\E}{\mathbb{E}}
\newcommand{\one}{\mathbf 1}
\newcommand{\dd}{\,\mathrm d}
\newcommand{\bd}{b_{\mathrm d}}
\newcommand{\boff}{b_{\mathrm{off}}}
\newcommand{\lo}{\mathrm{lo}}
\newcommand{\hi}{\mathrm{hi}}
\newcommand{\core}{\mathscr D_{\M}}
\newcommand{\norm}[2][]{\lVert #2\rVert_{#1}}
\DeclareMathOperator{\supp}{supp}

\DeclareMathOperator{\pv}{p.v.}

\makeatletter
\@namedef{subjclassname@2020}{%
  \textup{2020} Mathematics Subject Classification}
\makeatother

\title[Operator-valued rough singular integrals]
{Weak type $(1,1)$ estimates for noncommutative\\ rough   singular integrals}

\author[Cao]{Junming Cao}
\address{Junming Cao, School of Mathematical Sciences, Huaqiao University, Quanzhou 362021, China} \email{caojunming@hqu.edu.cn}

\author[Li]{Wei Li}
\address{Wei Li, School of Mathematics and Statistics, Central South University, Changsha 410075, China}
\email{liwei.csu@csu.edu.cn}

\author[Wu]{Lian Wu}
\address{Lian Wu, School of Mathematics and Statistics, Central South University, Changsha 410075, China}
\email{wulian@csu.edu.cn}
\makeatletter
\@namedef{subjclassname@2020}{%
	\textup{2020} Mathematics Subject Classification}
\makeatother
 

\begin{document}

\begin{abstract}
 
We establish the weak type \((1,1)\) endpoint estimate for operator-valued rough  singular integrals with kernels satisfying the classical \(L\log L\) integrability and cancellation conditions. This extends Seeger's endpoint theorem [{J. Amer. Math. Soc.} \textbf{9} (1996)] to the noncommutative setting. A key ingredient is an off-diagonal analysis that controls interactions between different kernel scales and avoids an additional logarithmic loss. 
	\end{abstract}
	
	\maketitle
	
	\section{Introduction}
	Rough homogeneous singular integrals form a basic class of singular
	integral operators that fall outside the scope of the standard
	Calder\'on--Zygmund theory, since no regularity is imposed on the angular part of the kernel.
Let $\Omega\in L_1(S^{d-1})$ be complex-valued and extend it
homogeneously of degree zero to $\R^d\setminus\{0\}$. We consider
	\begin{equation}\label{eq:mainoperator}
		T_\Omega f(x)=\pv\int_{\R^d}
		\frac{\Omega(y/|y|)}{|y|^d}f(x-y)\dd y,
		\qquad d\geq2,
	\end{equation}
	under the cancellation condition
	\begin{equation}\label{eq:meanzero}
		\int_{S^{d-1}}\Omega(\theta)\dd\sigma(\theta)=0.
	\end{equation}
	Calder\'on and Zygmund \cite{CZ1956} established the fundamental
	$L_p$ boundedness results for rough homogeneous singular integrals $	T_\Omega $
	using the method of rotations.
	
	The weak endpoint theory is substantially more delicate. In dimension
	two, Christ and Rubio de Francia \cite{ChristRubio} proved the weak
	type $(1,1)$ boundedness of $T_\Omega$ under the condition
	$\Omega\in L(\log^+L)(S^{1})$, and meanwhile  Hofmann \cite{Hofmann}
	obtained a corresponding result under a slightly stronger assumption.
	In arbitrary dimension, Seeger \cite{Seeger} established the weak
	type $(1,1)$ estimate under the conditions \eqref{eq:meanzero} and
	$\Omega\in L(\log^+L)(S^{d-1})$. Tao \cite{Tao} subsequently extended
	the result to homogeneous convolution operators on
	homogeneous groups.
	
	A closely related problem concerns the less singular maximal operator
	\begin{equation}\label{eq-maximal-rough}
		M_\Omega f(x)
		=
		\sup_{r>0}r^{-d}
		\int_{|y|\leq r}
		\left|
		\Omega\left(\frac{y}{|y|}\right)f(x-y)
		\right|\dd y,
	\end{equation}
	which is a variant of the Hardy--Littlewood maximal function.
	Christ \cite{Christ1988} obtained a weak type $(1,1)$ estimate for
	\eqref{eq-maximal-rough} in dimension two under the assumption
	$\Omega\in L^q(S^1)$ for some $q>1$. Christ and Rubio de Francia
	\cite{ChristRubio} subsequently established the weak type $(1,1)$
	bound in higher dimensions under
	$\Omega\in L(\log^+L)(S^{d-1})$.
	
	The weak endpoint estimates for \(M_\Omega\) and \(T_\Omega\) rely on closely related ideas, but their proofs present distinct technical challenges. For the rough maximal operator, Christ and Rubio de Francia \cite{ChristRubio}  combined a \(TT^*\)-type argument with a depth-dependent geometric analysis to establish regularity estimates for the kernels of the resulting composed operators.	
   The passage from the maximal operator to the singular
	integral introduces an additional difficulty: the supremum over
	radii is replaced by a summation over dyadic scales. As a result, the \(L^2\) analysis involves cross terms between different scales. These terms do not arise in the corresponding argument for the maximal operator and make the analysis substantially more intricate in higher dimensions. To handle these
	interactions,  
	Seeger \cite{Seeger} developed a
	microlocal decomposition of the kernel and combined it with
	Fourier-analytic estimates.

	Over the past two decades, noncommutative harmonic analysis has developed
	substantially, with advances in noncommutative Littlewood--Paley theory,
	Fourier multiplier theory, maximal inequalities, and singular integral theory;
	see, for instance,
	\cite{PisierXu2003,Mei,MeiParcet,ChenXuYin,JungeMeiParcet,GonzalezJungeParcet}.
	A foundational contribution in the endpoint theory of noncommutative singular integrals
	was made by Parcet \cite{Parcet}, who introduced a noncommutative
	Calder\'on--Zygmund decomposition based on Cuculescu projections \cite{Cucu}
	and combined it with a pseudo-localization principle to establish weak type
	$(1,1)$ estimates for standard operator-valued Calder\'on--Zygmund operators.
	In his original decomposition, both the good and bad parts contain
	off-diagonal terms, and pseudo-localization is used to control those in the good part. Cadilhac \cite{Cadilhac2018} subsequently gave a
	simplified proof of Parcet's weak endpoint theorem. A further structural refinement was introduced by Cadilhac,
	Conde-Alonso and Parcet \cite{CCP}. Their improved
	Calder\'on--Zygmund decomposition takes the form
	\[
	f=g+\bd+\boff,
	\]
	where the good part is bounded and all off-diagonal terms are contained in \(b_{\mathrm{off}}\). The good part can therefore be treated directly using an \(L^2\) estimate, leaving the off-diagonal analysis entirely to \(b_{\mathrm{off}}\).
	This refinement eliminates the need for the pseudo-localization step used to control the off-diagonal good terms in Parcet’s original argument. The same regular decomposition is used by Hong, Lai and Xu
	in their work on maximal singular integrals \cite{HLX}.

	For  rough homogeneous singular integrals, Hong, Lai and Xu
	\cite[Theorem~1.5(i)]{HLX} proved the operator-valued maximal strong type $(p,p)$ estimates,
	$1<p<\infty$, under the   assumption  $\Omega\in L(\log^+L)(S^{d-1})$. Their maximal weak type $(1,1)$ result
	\cite[Theorem~1.5(ii)]{HLX} assumes
	$\Omega\in L_2(S^{d-1})$ together with the $L_2$-Dini condition. Lai \cite{Lai}  studied the operator-valued analogue of  the rough maximal operator $M_\Omega$ and proved a
	maximal weak type $(1,1)$ estimate under
	$
	\Omega\in L(\log^+L)^2(S^{d-1}).
	$
	His proof combines the noncommutative Calder\'on--Zygmund
	decomposition introduced by Parcet with a microlocal analysis
	inspired by Seeger. A Rademacher linearization reduces the problem to
	an operator of the form
	$
	\sum_j\varepsilon_jT_j,
	$
	where $(\varepsilon_j)$ is a Rademacher sequence. Orthogonality in the
	auxiliary probability variable converts the corresponding $L_2$ norm
	into a square sum over the kernel scales.
	
	The purpose of this paper is to establish the classical $L\log L$ endpoint
	for the operator-valued principal-value singular integral. For $\Omega\neq0$, put
	\begin{equation}\label{eq:angularconstant}
		C_{\Omega,1}
		=
		\int_{S^{d-1}}
		|\Omega(\theta)|
		\left(
		1+\log^+\frac{|\Omega(\theta)|}{\|\Omega\|_1}
		\right)
		\dd\sigma(\theta),
	\end{equation}
	and set $C_{0,1}=0$. The quantity in \eqref{eq:angularconstant} is homogeneous in
	$\Omega$, and its finiteness is equivalent to
	$\Omega\in L(\log^+L)(S^{d-1})$.
	
	Let $(\M,\tau)$ be a semifinite von Neumann algebra equipped with a
	normal semifinite faithful trace and set
	$
	\A=L_\infty(\R^d)\,\overline\otimes\,\M,
$ $
	\varphi=\int_{\R^d}\dd x\otimes\tau.
$
	We use the same notation $T_\Omega$ for the operator-valued
	singular integral defined on
	$
	\core=C_c^\infty(\R^d)\otimes(\M\cap L_1(\M))
$
	by the formula in \eqref{eq:mainoperator}.
	The main theorem is the following.
\begin{theorem}\label{thm:main}
	Suppose that $\Omega\in L(\log^+L)(S^{d-1})$ satisfies
	\eqref{eq:meanzero}. Then $T_\Omega$ extends uniquely to a bounded
	linear map
	\[
	T_\Omega:L_1(\A)\longrightarrow L_{1,\infty}(\A).
	\]
	Moreover,
	\begin{equation}\label{eq-main-rough}
		\norm[L_{1,\infty}(\A)]{T_\Omega f}
		\leq C_d C_{\Omega,1}\norm[L_1(\A)]{f},
	\end{equation}
	where $C_d$ depends only on the dimension.
\end{theorem}

\begin{remark}\label{rem:nonhomogeneous}
	For $\M=\C$, Theorem~\ref{thm:main} recovers
	the homogeneous case of Seeger's endpoint theorem
	\cite{Seeger}. We also extend the result to the
	nonhomogeneous convolution kernels considered
	there; see Section~\ref{sec:nonhomogeneous}.
\end{remark}

	The proof combines two forms of the noncommutative
	Calder\'on--Zygmund decomposition with Seeger's microlocal
	decomposition. We begin with the regular decomposition of Cadilhac,
	Conde-Alonso and Parcet, whose bounded good part can be handled
	directly by the $L_2$ boundedness of $T_\Omega$. The remaining
	off-diagonal bad part is further resolved into the blocks
	$b_{k,\ell}$ and $g_{k,\ell}$ from Parcet's earlier decomposition;
	see Subsection~\ref{subsec-b-g}. These two families have different
	cancellation and norm properties. We now outline how these two families are used in the estimate of
	the off-diagonal bad part.
	
	We first decompose the kernel $\Omega$ into two parts; see
	Subsection~\ref{subsec-kernel-decom}. For the terms involving the
	first part of $\Omega$, the $b_{k,\ell}$ and $g_{k,\ell}$ blocks are
	treated separately. For the $g_{k,\ell}$ blocks, we use their cancellation at the
	corresponding dyadic scale together with the Fourier decay of the
	kernel; see Lemma~\ref{lem:martingale}. For the $b_{k,\ell}$ blocks, the main
	difficulty is that the sum over the kernel scales has to be estimated
	as a whole. In particular, we need to control the interactions
	between different kernel scales. Nearby scales are handled by a
	simple counting argument, while more widely separated scales are
	treated by a $TT^*$-type argument combining geometric kernel
	estimates with the structure of the Cuculescu projections; see
	Lemma~\ref{lem:offprincipal}.
	
	For the terms involving the second part of $\Omega$, in contrast, it
	is important to keep the $b_{k,\ell}$ and $g_{k,\ell}$ blocks
	together. For each fixed $\theta\in S^{d-1}$, only finitely many
	scale levels can occur, and their number is controlled by a
	logarithmic term involving $\Omega$. If the contributions corresponding to each pair of scales 
	were estimated separately and then summed, an additional logarithmic
	factor would appear. To avoid this loss, we keep the two families
	together and estimate their sum directly in $L_1$; see
	Lemma~\ref{lem:cumulative}. The resulting estimate grows only
	linearly with the number of relevant scale levels. After integration
	in $\theta$, this gives exactly the $L(\log^+L)$ quantity appearing in
	$C_{\Omega,1}$.

					Shortly before the completion of this manuscript, Lai \cite{Lai2026}
					posted a related preprint proving the weak type $(1,1)$ boundedness
					of the same noncommutative rough singular integral under the stronger
					assumption
					$
					\Omega\in L(\log^+L)^{5/2}(S^{d-1}).
					$
					Lai's proof is based on Parcet's original noncommutative
					Calder\'on--Zygmund decomposition, whereas our argument combines the
					noncommutative Calder\'on--Zygmund decomposition of
					Cadilhac--Conde-Alonso--Parcet with the off-diagonal structure of
					Parcet's original decomposition.
										We note that the proof of Lemma~\ref{lem:offprincipal} is closely
					related to Lai's proof of Proposition~3.8 in \cite{Lai2026}. Both
					arguments combine a microlocal decomposition with a $TT^*$ method,
					while the main differences lie in the treatment of the single-scale
					terms and the far interactions. A further difference arises in the
					treatment of the second part of $\Omega$: we keep the two
					off-diagonal families $b_{k,\ell}$ and $g_{k,\ell}$ together and
					apply the cumulative $L_1$ estimate in Lemma~\ref{lem:cumulative}.
					This avoids an additional logarithmic loss in the summation
					and is a key step in proving the weak type $(1,1)$ estimate
					under the classical assumption
					$\Omega\in L(\log^+L)(S^{d-1})$.

		The paper is organized as follows.
	Section~\ref{sec-pre} introduces the notation, kernel decompositions,
	and noncommutative Calder\'on--Zygmund decompositions used throughout
	the paper, including the proofs of some necessary lemmas.
	Section~\ref{sec:plan} proves Theorem~\ref{thm:main} assuming four
	auxiliary lemmas.
	Section~\ref{sec:diagproof} proves the first of the four lemmas,
	which concerns the diagonal bad function.
	Section~\ref{sec:offproof} proves the remaining three lemmas,
	which concern the off-diagonal bad function.
	Section~\ref{sec:nonhomogeneous} extends Theorem~\ref{thm:main} to nonhomogeneous convolution kernels.

		Throughout the paper, constants denoted by $C$ or $C_d$ are positive
	and may change from line to line. We write
	$A\lesssim B$ if $A\leq CB$, and $A\lesssim_d B$ if the constant
	depends only on the dimension $d$.

	\section{Preliminaries  }\label{sec-pre}

\subsection{Noncommutative integration}

Throughout the paper, let $(\M,\tau)$ be a semifinite von Neumann algebra equipped with a
normal semifinite faithful trace. Set
\[
    \A=L_\infty(\R^d)\,\overline\otimes\,\M,
    \qquad
    \varphi=\int_{\R^d}\dd x\otimes\tau.
\]
For $1\leq p<\infty$, we use the standard isometric identification
\[
    L_p(\A)=L_p(\R^d;L_p(\M))
\]
and write
\[
    \norm[L_p(\A)]{h}^p
    =\varphi(|h|^p)
    =\int_{\R^d}\norm[L_p(\M)]{h(x)}^p\dd x.
\]
We set $L_\infty(\A)=\A$ and write $\norm[p]{\cdot}$ whenever the
ambient algebra is clear.
The space $L_2(\A)$ is a Hilbert space with inner product
\[
\langle x,y\rangle_2=\varphi(y^*x).
\]
For a $\varphi$-measurable operator $h$ and $t>0$, define
\[
    d_t(h)=\varphi\bigl(\one_{(t,\infty)}(|h|)\bigr),
    \qquad
    \norm[1,\infty]{h}=\sup_{t>0}t\,d_t(h).
\]

We record the standard estimates used below.

\begin{lemma}\label{lem:basic-integration}
    The following assertions hold.
    \begin{enumerate}[label=\textup{(\roman*)},leftmargin=2.2em]
        \item For $s,t>0$,
        \[
            d_{s+t}(x+y)\leq d_s(x)+d_t(y).
        \]

        \item If $h\in L_p(\A)$ and $t>0$, then
        \[
            d_t(h)\leq t^{-p}\norm[p]{h}^p.
        \]

        \item If $K\in L_1(\R^d)$ is scalar and $h\in L_p(\A)$,
        $1\leq p<\infty$, then
        \begin{equation}\label{eq-Young}
            \norm[p]{K*h}
            \leq
            \norm[L_1(\R^d)]{K}\norm[p]{h}.
        \end{equation}

        \item With
        $
            \widehat h(\xi)
            =\int_{\R^d}e^{-i\langle x,\xi\rangle}h(x)\dd x,
        $
          Plancherel's theorem gives
        \[
            \norm[2]{h}^2
            =(2\pi)^{-d}
            \int_{\R^d}
            \norm[L_2(\M)]{\widehat h(\xi)}^2\dd\xi.
        \]
    \end{enumerate}
\end{lemma}

Let $(\Sigma,\mu)$ be a measure space. We shall also use the column
H\"older inequality \cite[Proposition~1.1]{Mei}.

\begin{lemma}\label{lem:column-holder}
      Let $1\leq p,q,r\leq\infty$ satisfy
        $1/r=1/p+1/q$. For
        $a\in L_p(\M;L_2^c(\Sigma))$ and
        $b\in L_q(\M;L_2^c(\Sigma))$,
        \[
            \left\|\int_\Sigma a^*(\omega)b(\omega)\dd\mu(\omega)\right\|_r
            \leq
            \left\|\left(\int_\Sigma |a(\omega)|^2\dd\mu(\omega)\right)^{1/2}\right\|_p
            \left\|\left(\int_\Sigma |b(\omega)|^2\dd\mu(\omega)\right)^{1/2}\right\|_q.
        \]

  \end{lemma}
         In particular, for $(p,q,r)=(\infty,2,2)$,
        \begin{equation}\label{eq:column22}
            \left\|\int_\Sigma a^*b\dd\mu\right\|_{L_2(\M)}^2
            \leq
            \left\|\int_\Sigma a^*a\dd\mu\right\|_{L_\infty(\M)}
            \tau\left(\int_\Sigma b^*b\dd\mu\right).
        \end{equation}

\subsection{The dyadic martingales}

For $k\in\Z$, let $\Q_k$ be the family of standard dyadic cubes in
$\R^d$ of side length $2^{-k}$, and let $\Sigma_k$ be the
$\sigma$-algebra generated by $\Q_k$. Set
\[
    \A_k=L_\infty(\R^d,\Sigma_k)\,\overline\otimes\,\M.
\]
Then $(\A_k)_{k\in\Z}$ is an increasing filtration of $\A$. Let
\[
    \E_k:\A\longrightarrow\A_k
\]
be the trace-preserving conditional expectation, extended
contractively to $L_p(\A)$ for $1\leq p<\infty$. For a locally
integrable operator-valued function $f$,
\[
    \E_kf=\sum_{Q\in\Q_k}f_Q\one_Q,
    \qquad
    f_Q=\frac1{|Q|}\int_Qf(y)\dd y.
\]
We write
\[
    f_k=\E_kf,
    \qquad
    df_k=f_k-f_{k-1}.
\]
The conditional expectations satisfy the following standard properties.
\begin{enumerate}[label=\textup{(\roman*)},leftmargin=2.2em]
    \item For $k,l\in\Z$,
    \[
        \E_k\E_l=\E_{\min\{k,l\}}.
    \]

   \item If $a,b\in\A_k$ and $h\in L_p(\A)$, $1\leq p<\infty$, then
   \[
   \E_k(ahb)=a\,\E_k(h)\,b.
   \]
\end{enumerate}

\subsection{Dyadic kernels and Seeger's microlocal decomposition}

Choose a fixed real nonnegative $\beta\in C_c^\infty((1/2,2))$ such that
\[
    \sum_{j\in\Z}\beta(2^{-j}r)^2=1,
    \qquad r>0,
\]
as in \cite[Section~3]{Seeger}. For $\Psi\in L_1(S^{d-1})$, we use
the same symbol $\Psi$ for its homogeneous extension of degree zero to
$\R^d\setminus\{0\}$. For $j\in\Z$, define
\begin{equation}\label{eq:dyadickernels}
    K_j^\Psi(x)
    =\beta(2^{-j}|x|)^2\frac{\Psi(x/|x|)}{|x|^d},
    \qquad
    T_j^\Psi h=K_j^\Psi*h.
\end{equation}
Throughout the paper, $I$ denotes a finite subset of $\Z$, and we set
\[
    T_{\Omega,I}f=\sum_{j\in I}T_j^\Omega f.
\]
Polar coordinates give
\begin{equation}\label{eq:kernelL1}
    \begin{aligned}
        \norm[L_1(\R^d)]{K_j^\Psi}
        &=\int_{S^{d-1}}|\Psi(\theta)|\dd\sigma(\theta)
          \int_0^\infty\beta(2^{-j}r)^2\frac{\dd r}{r}\\
        &=\norm[1]{\Psi}\int_0^\infty\beta(r)^2\frac{\dd r}{r}
        \lesssim_d\norm[1]{\Psi}.
    \end{aligned}
\end{equation}

We first record the Fourier estimates for the dyadic kernels
\eqref{eq:dyadickernels}.

\begin{lemma}\label{lem:fourier}
	The following assertions hold.
	\begin{enumerate}[label=\textup{(\roman*)},leftmargin=2.2em]
		\item For $\Psi\in L_\infty(S^{d-1})$, $j\in\Z$, and
		$\xi\neq0$,
		\begin{equation}\label{eq:boundedmultiplier}
			|\widehat{K_j^\Psi}(\xi)|
			\lesssim_d
			\|\Psi\|_\infty
			\min\{1,(2^j|\xi|)^{-1}\}.
		\end{equation}
		
		\item If $\Omega$ satisfies \eqref{eq:meanzero} and
		$C_{\Omega,1}<\infty$, then
		\begin{equation}\label{eq:absolutemultiplier}
			\sup_{\xi\neq0}
			\sum_{j\in\Z}
			|\widehat{K_j^\Omega}(\xi)|
			\lesssim_d C_{\Omega,1}.
		\end{equation}
		
		\item Consequently,
		\[
		\sup_I
		\|T_{\Omega,I}\|_{L_2(\A)\to L_2(\A)}
		\lesssim_d C_{\Omega,1}.
		\]
	 	\end{enumerate}
\end{lemma}
The proof is standard; we nevertheless provide it for the reader's convenience.
\begin{proof}
	Set
	\[
	\Phi(t)
	=
	\int_0^\infty
	\frac{\beta(r)^2}{r}e^{-irt}\dd r.
	\]
	Since $\beta(r)^2/r$ is smooth and compactly supported in
	$(1/2,2)$, integration by parts gives, for every $N\geq0$,
	\[
	|\Phi(t)|
	\leq C_N(1+|t|)^{-N},
	\qquad
	|\Phi(t)-\Phi(0)|
	\leq C|t|.
	\]
	By polar coordinates,
	\begin{equation}\label{eq:FourierPhi}
		\widehat{K_j^\Psi}(\xi)
		=
		\int_{S^{d-1}}
		\Psi(\theta)
		\Phi\bigl(2^j\langle\theta,\xi\rangle\bigr)
		\dd\sigma(\theta).
	\end{equation}
	
	We first prove \eqref{eq:boundedmultiplier}.
	If $2^j|\xi|\leq1$, then \eqref{eq:FourierPhi} immediately gives
	\[
	|\widehat{K_j^\Psi}(\xi)|
	\lesssim_d\|\Psi\|_\infty.
	\]
	If $2^j|\xi|>1$, take $N=2$. Since
	\[
	\int_{S^{d-1}}
	\bigl(
	1+a|\langle\theta,\xi/|\xi|\rangle|
	\bigr)^{-2}
	\dd\sigma(\theta)
	\lesssim_d a^{-1},
	\qquad a\geq1,
	\]
	we obtain
	\[
	|\widehat{K_j^\Psi}(\xi)|
	\lesssim_d
	\|\Psi\|_\infty(2^j|\xi|)^{-1}.
	\]
	This proves \eqref{eq:boundedmultiplier}.
	
	We next prove \eqref{eq:absolutemultiplier}.
	The case $\Omega=0$ is immediate, so assume
	$\|\Omega\|_1>0$. By \eqref{eq:meanzero},
	\[
	\widehat{K_j^\Omega}(\xi)
	=
	\int_{S^{d-1}}
	\Omega(\theta)
	\left[
	\Phi\bigl(2^j\langle\theta,\xi\rangle\bigr)
	-\Phi(0)
	\right]
	\dd\sigma(\theta).
	\]
	Hence, when $2^j|\xi|\leq1$,
	\begin{equation}\label{eq:Fourier-low}
		|\widehat K_j^\Omega(\xi)|
		\lesssim
		\|\Omega\|_1\,2^j|\xi|.
	\end{equation}
	Therefore
	\[
	\sum_{j:\,2^j|\xi|\leq1}
	|\widehat K_j^\Omega(\xi)|
	\lesssim
	\|\Omega\|_1.
	\]
	
	Suppose now that $2^j|\xi|>1$. Split
	\[
	\Omega
	=
	\Omega\one_{\{
		|\Omega|\leq
		\|\Omega\|_1(2^j|\xi|)^{1/2}\}}
	+
	\Omega\one_{\{
		|\Omega|>
		\|\Omega\|_1(2^j|\xi|)^{1/2}\}}.
	\]
	Applying \eqref{eq:boundedmultiplier} to the first part and
	\eqref{eq:kernelL1} to the second gives
	\begin{equation}\label{eq:Fourier-high}
		|\widehat K_j^\Omega(\xi)|
		\lesssim_d
		\|\Omega\|_1(2^j|\xi|)^{-1/2}
		+
		\int_{\{
			|\Omega|>
			\|\Omega\|_1(2^j|\xi|)^{1/2}\}}
		|\Omega|\dd\sigma.
	\end{equation}
	The first term on the right is summable:
	\[
	\sum_{j:\,2^j|\xi|>1}
	\|\Omega\|_1(2^j|\xi|)^{-1/2}
	\lesssim
	\|\Omega\|_1.
	\]
	For the second term, changing the order of summation and integration,
	\[
	\begin{aligned}
		&\sum_{j:\,2^j|\xi|>1}
		\int_{\{
			|\Omega|>
			\|\Omega\|_1(2^j|\xi|)^{1/2}\}}
		|\Omega|\dd\sigma\\
		&\quad=
		\int_{S^{d-1}}
		|\Omega(\theta)|
		\#\left\{
		j:
		1<2^j|\xi|
		<
		\frac{|\Omega(\theta)|^2}{\|\Omega\|_1^2}
		\right\}
		\dd\sigma(\theta)\\
		&\quad\lesssim
		\int_{S^{d-1}}
		|\Omega(\theta)|
		\left(
		1+\log^+
		\frac{|\Omega(\theta)|}{\|\Omega\|_1}
		\right)
		\dd\sigma(\theta)\\
		&\quad\lesssim
		C_{\Omega,1}.
	\end{aligned}
	\]
	Combining this with \eqref{eq:Fourier-low} and
	\eqref{eq:Fourier-high} proves
	\eqref{eq:absolutemultiplier}.
	
	Finally, by Plancherel's theorem,
	\[
	\begin{aligned}
		\|T_{\Omega,I}h\|_2
	 \leq
		\sup_{\xi\neq0}
		\left|
		\sum_{j\in I}
		\widehat K_j^\Omega(\xi)
		\right|
		\|h\|_2 
	 		 \lesssim_d
		C_{\Omega,1}\|h\|_2.
	\end{aligned}
	\]
	This proves the last assertion.
\end{proof}

\begin{remark}
	By \eqref{eq:absolutemultiplier}, Plancherel's theorem and
	dominated convergence imply that, for every $h\in L_2(\A)$,
	the sequence $T_{\Omega,[-J,J]}h$ converges in $L_2(\A)$ as
	$J\to\infty$.
		For
	$
	f\in
	C_c^\infty(\R^d)\otimes(\M\cap L_1(\M)),
	$
	this limit agrees with the principal-value operator
	\eqref{eq:mainoperator}; see \cite[Section~1]{Seeger}.
	We therefore denote the $L_2$-limit by $T_\Omega h$. In particular,
	\begin{equation}\label{eq:L2-convergence}
		T_{\Omega,[-J,J]}h
		\longrightarrow
		T_\Omega h
		\qquad\text{in }L_2(\A),
		\qquad h\in L_2(\A),
	\end{equation}
	and
	\[
	\|T_\Omega h\|_2
	\lesssim_d
	C_{\Omega,1}\|h\|_2.
	\]
\end{remark}

For the microlocal decomposition we follow \cite[Section~2]{Seeger}.
Throughout the paper, set $n_0=100$. For every integer $s\geq n_0$,
choose points $(e_\nu^s)_\nu\subset S^{d-1}$ and a measurable
partition $(E_\nu^s)_\nu$ of $S^{d-1}$, with $e_\nu^s\in E_\nu^s$,
such that the following properties hold.
\begin{enumerate}[label=\textup{(\roman*)},leftmargin=2.2em]
    \item
    \[
        \operatorname{diam}(E_\nu^s)\leq2^{-s},
        \qquad
        |e_\nu^s-e_{\nu'}^s|\geq c_d2^{-s}
        \quad(\nu\ne\nu').
    \]

    \item
    \begin{equation}\label{eq:angularpartition}
        \#\{\nu\}\lesssim_d2^{s(d-1)}.
    \end{equation}
\end{enumerate}
Choose $\psi\in C_c^\infty(\R)$ such that
\[
    \psi(t)=1\quad(|t|\leq2),
    \qquad
    \supp\psi\subset[-4,4].
\]
For each index $\nu$, define the Fourier multiplier $P_\nu^s$ by
\[
    \widehat{P_\nu^s h}(\xi)
    =\psi\left(
        2^{s/2}\left\langle e_\nu^s,\frac{\xi}{|\xi|}\right\rangle
      \right)\widehat h(\xi),
    \qquad \xi\ne0.
\]
For $j\in\Z$ and $\Psi\in L_\infty(S^{d-1})$, define
\begin{equation}\label{eq:Seegerkernels}
    H_{j\nu}^{s,\Psi}(x)
    =K_j^\Psi(x)
      \one_{E_\nu^s}\left(\frac{x}{|x|}\right),
    \qquad
    \Gamma_j^{s,\Psi}=\sum_\nu P_\nu^sH_{j\nu}^{s,\Psi}.
\end{equation}
Thus Seeger's microlocal decomposition takes the form
\begin{equation}\label{eq:Seegersplitting}
    K_j^\Psi
    =\Gamma_j^{s,\Psi}
     +\bigl(K_j^\Psi-\Gamma_j^{s,\Psi}\bigr),
\end{equation}
as in \cite[(2.2)]{Seeger}.

We conclude this subsection with two estimates associated with the
microlocal decomposition.
\begin{lemma}\label{lem:overlap}
	 Let $s\geq n_0$ be an integer.  For every finite family $U_\nu\in L_2(\A)$,
		\[
		\left\|\sum_\nu P_\nu^sU_\nu\right\|_2^2
		\lesssim_d2^{s(d-3/2)}\sum_\nu\norm[2]{U_\nu}^2.
		\]
	\end{lemma}
	
	\begin{proof}
		Seeger's count \cite[(2.3)--(2.4)]{Seeger}, with $\kappa=1/2$, gives
		\[
		\sup_{\xi\ne0}\#\{\nu:\psi\bigl(2^{s/2}\langle e_\nu^s,\xi/|\xi|\rangle\bigr)\ne0\}
		\leq C_d2^{s(d-3/2)}.
		\]
		Indeed, a nonzero multiplier requires
		$|\langle e_\nu^s,\xi/|\xi|\rangle|\leq4\,2^{-s/2}$, and the
		points $e_\nu^s$ are separated by at least $c_d2^{-s}$.
		Thus $\sum_\nu|\psi\bigl(2^{s/2}\langle e_\nu^s,\xi/|\xi|\rangle\bigr)|^2\leq C_d2^{s(d-3/2)}$.
	The Cauchy--Schwarz inequality  and Plancherel's theorem now imply
		\[
		\begin{aligned}
			\left\|\sum_\nu P_\nu^sU_\nu\right\|_2^2
			&=(2\pi)^{-d}\int
			\left\|\sum_\nu\psi\bigl(2^{s/2}\langle e_\nu^s,\xi/|\xi|\rangle\bigr)\widehat U_\nu(\xi)\right\|_{L_2(\M)}^2\dd\xi\\
			&\leq(2\pi)^{-d}\int\sum_\nu|\psi\bigl(2^{s/2}\langle e_\nu^s,\xi/|\xi|\rangle\bigr)|^2
			\sum_\nu\norm[L_2(\M)]{\widehat U_\nu(\xi)}^2\dd\xi\\
			&\leq C_d2^{s(d-3/2)}\sum_\nu\norm[2]{U_\nu}^2.
		\end{aligned}
		\]	
	\end{proof}
	The following lemma is the operator-valued version of
	\cite[Lemma~2.2]{Seeger}.
	
	\begin{lemma}\label{lem:remainder}
		Let $s\geq n_0$  be an integer  and $j\in\mathbb Z$. Let
		$h_Q\in L_1(\A)$ be supported in a cube $Q$ of side length
		$2^{j-s}$ and satisfy
		\[
		\int_Q h_Q(x)\dd x=0.
		\]
		Then, for every $\Psi\in L_\infty(S^{d-1})$,
		\begin{equation}\label{eq:remainder}
			\|(K_j^\Psi-\Gamma_j^{s,\Psi})*h_Q\|_{L_1(\A)}
			\lesssim_d
			\|\Psi\|_\infty\,2^{-s/4}\|h_Q\|_{L_1(\A)}.
		\end{equation}
	\end{lemma}

	\begin{proof}
		Set
		\[
		R_j^{s,\Psi}
		=
		K_j^\Psi-\Gamma_j^{s,\Psi},
		\]
		and let $c_Q$ be the center of $Q$. Since $h_Q$ is supported in
		$Q$ and satisfies
		\[
		\int_Q h_Q(y)\dd y=0,
		\]
		we may write
		\[
		R_j^{s,\Psi}*h_Q(x)
		=
		\int_Q
		\bigl[
		R_j^{s,\Psi}(x-y)
		-
		R_j^{s,\Psi}(x-c_Q)
		\bigr]
		h_Q(y)\dd y.
		\]
		By Minkowski's inequality,
		\[
		\begin{aligned}
			\|R_j^{s,\Psi}*h_Q\|_{L_1(\A)}
			&\leq
			\int_Q
			\|h_Q(y)\|_{L_1(\M)}
			\int_{\R^d}
			\left|
			R_j^{s,\Psi}(x-y)
			-
			R_j^{s,\Psi}(x-c_Q)
			\right|
			\dd x\dd y\\
			&\leq
			\sup_{y\in Q}
			\int_{\R^d}
			\left|
			R_j^{s,\Psi}(x-y)
			-
			R_j^{s,\Psi}(x-c_Q)
			\right|
			\dd x\,
			\|h_Q\|_{L_1(\A)}.
		\end{aligned}
		\]
		It therefore remains to estimate the scalar kernel difference.
		
		The proof of \cite[Lemma~2.2]{Seeger}, with $\kappa=1/2$, yields
		\[
		\sup_{y\in Q}
		\int_{\R^d}
		\left|
		R_j^{s,\Psi}(x-y)-R_j^{s,\Psi}(x-c_Q)
		\right|\dd x
		\lesssim_{N,d}
		M_0\,2^{-s\varepsilon}
		+
		M_N\,2^{s(d+(\varepsilon-\frac12)N)},
		\]
		for $0\leq\varepsilon\leq1$ and  every integer  $N\geq d+1$, where $M_N$ is the
		quantity appearing in \cite[(2.1)]{Seeger}. In the present setting,
		\[
		K_j^\Psi(r\theta)
		=
		\beta(2^{-j}r)^2\Psi(\theta)r^{-d},
		\]
		so that
		\[
		M_N\lesssim_{N,d}\|\Psi\|_\infty.
		\]
		Taking
		$
		\varepsilon=\frac14,
		N=4d+4,
		$
		we have
		$
		d+\left(\frac14-\frac12\right)N=-1.
		$
		Hence
		\[
		\begin{aligned}
			&\sup_{y\in Q}
			\int_{\R^d}
			\left|
			R_j^{s,\Psi}(x-y)
			-
			R_j^{s,\Psi}(x-c_Q)
			\right|\dd x\\
			&\qquad\lesssim_d
			\|\Psi\|_\infty
			\bigl(2^{-s/4}+2^{-s}\bigr)
			\lesssim_d
			\|\Psi\|_\infty2^{-s/4}.
		\end{aligned}
		\]
		Substituting this estimate into the preceding inequality gives
		\[
		\|(K_j^\Psi-\Gamma_j^{s,\Psi})*h_Q\|_{L_1(\A)}
		\lesssim_d
		\|\Psi\|_\infty\,2^{-s/4}
		\|h_Q\|_{L_1(\A)},
		\]
		which proves the lemma.
	\end{proof}

\subsection{A decomposition of the kernels}\label{subsec-kernel-decom}

Fix $\delta=1/16$. For every integer $s\geq n_0$, following
\cite[Section~3]{Seeger} and \cite[Section~3B]{Lai}, set
\[
    \Omega_s^\lo
    =\Omega\one_{\{|\Omega|\leq2^{\delta s}\|\Omega\|_1\}},
    \qquad
    \Omega_s^\hi
    =\Omega\one_{\{|\Omega|>2^{\delta s}\|\Omega\|_1\}}.
\]
The following properties will be used repeatedly.
\begin{enumerate}[label=\textup{(\roman*)},leftmargin=2.2em]
    \item
   $
        \Omega=\Omega_s^\lo+\Omega_s^\hi.
    $

    \item
    $
        \|\Omega_s^\lo\|_\infty
        \leq2^{\delta s}\|\Omega\|_1.
    $

    \item By the linearity of $\Psi\mapsto K_j^\Psi$ and
    \eqref{eq:Seegersplitting},
    \begin{equation}\label{eq:amplitudekernels}
        K_j^\Omega
        =\Gamma_j^{s,\Omega_s^\lo}
        +\bigl(K_j^{\Omega_s^\lo}
        -\Gamma_j^{s,\Omega_s^\lo}\bigr)
        +K_j^{\Omega_s^\hi}.
    \end{equation}
\end{enumerate}

\subsection{Noncommutative Calder\'on--Zygmund decomposition}

In this subsection, we first work under the assumptions
\begin{equation}\label{eq:finitemodel}
    f\in L_1(\A)_+\cap\A,
    \qquad
    \supp_{\R^d}f\ \text{is compact},
    \qquad
    f=\E_Mf\quad\text{for some }M\in\Z.
\end{equation}
We begin with Cuculescu's construction; see
\cite[Lemma~3.1]{Parcet} and \cite[Section~1]{CCP}.

\begin{lemma}\label{lem:cuculescu}
 Let $f$ satisfy \eqref{eq:finitemodel}	and  $\alpha>0$. There exists a decreasing sequence of projections
	$(q_k)_{k\in\Z}$ such that $q_k=1$ for all sufficiently negative
	$k$, and, for the remaining $k$,
	\[
	q_k
	=
	\one_{[0,\alpha]}(q_{k-1}f_kq_{k-1})
	\quad\text{in }q_{k-1}\A q_{k-1}.
	\]
	Moreover, for every $k\in\Z$,
	\begin{enumerate}[label=\textup{(\roman*)},leftmargin=2.2em]
		\item $q_k\in\A_k$ and $
		q_kf_kq_k\leq\alpha q_k.
		$
				\item $q_k$ commutes with $q_{k-1}f_kq_{k-1}$;
				\item The following  assertion holds
			$$\alpha\varphi(1-q)\leq\norm[1]{f},$$
			where
			$q=\bigwedge_{k\in\Z}q_k.$
	\end{enumerate}
\end{lemma}
Set
\[
p_k=q_{k-1}-q_k.
\]
Since $q_k\in\A_k$, we may write
\[
q_k=\sum_{Q\in\Q_k}q_Q\one_Q,
\]
where each $q_Q$ is a projection in $\M$. If $\widehat Q$ denotes
the dyadic parent of $Q$, define
\[
\pi_Q=q_{\widehat Q}-q_Q,
\qquad Q\in\Q_k.
\]
Then
\[
p_k=q_{k-1}-q_k
=\sum_{Q\in\Q_k}\pi_Q\one_Q.
\]
The following standard properties of the projections $(p_k)_{k\in\Z}$
follow from Cuculescu's construction and the dyadic regularity of the filtration, which can be found in
\cite{Parcet}.

\begin{lemma}\label{lem:pk-properties}
 Let $f$ satisfy \eqref{eq:finitemodel}.	For the projections defined above, the following assertions hold.
	\begin{enumerate}[label=\textup{(\roman*)},leftmargin=2.2em]
		\item
		For every $k\in\Z$,
		$
		p_k\in\A_k.
		$
		Moreover, the projections $(p_k)_{k\in\Z}$ are pairwise
		orthogonal and
		\[
		\sum_{k\in\Z}p_k=1-q.
		\]
		
		\item
		For every $k\in\Z$,
		\[
		p_kf_kq_k=q_kf_kp_k=0.
		\]
		
		\item
		We have
		\begin{equation}\label{eq:massbudget}
			\alpha\varphi(p_k)\leq\varphi(fp_k),
			\qquad
			\sum_{k\in\Z}\varphi(fp_k)\leq\norm[1]{f}.
		\end{equation}
		
		\item
		The dyadic regularity gives
		\begin{equation}\label{eq:cubeproperties}
			p_kf_kp_k\leq2^d\alpha p_k,
			\qquad
			\pi_Qf_Q\pi_Q\leq2^d\alpha\pi_Q
			\quad(Q\in\Q_k).
		\end{equation}
	\end{enumerate}
\end{lemma}

The noncommutative Calder\'on--Zygmund decomposition of
\cite[(1.1) and Lemma~1.1]{CCP}, also stated in
\cite[Theorem~2.5]{HLX}, takes the following form.

\begin{theorem}\label{thm:regularcz}
 Let $f$ satisfy \eqref{eq:finitemodel} and let $\alpha>0$.
Let $(q_k)_k$, $(p_k)_k$, and $q$ be the projections introduced above.
Then
	\begin{equation}\label{eq:czdecomp}
		f=g+\bd+\boff,
	\end{equation}
	where the following assertions hold.
    \begin{enumerate}[label=\textup{(\roman*)},leftmargin=2.2em]
        \item
        $
            g=qfq+\sum_kp_kf_kp_k,
        $
        with
        \begin{equation}\label{eq:czbounds}
            \norm[1]{g}\leq\norm[1]{f},
            \qquad
            \norm[\infty]{g}\leq2^d\alpha.
        \end{equation}

        \item
        $
            \bd=\sum_kb_{\mathrm d,k},
           $  where $$
            b_{\mathrm d,k}=p_k(f-f_k)p_k,
        $$
        satisfying
        \[
            \E_kb_{\mathrm d,k}=0,
            \qquad
            \sum_k\norm[1]{b_{\mathrm d,k}}
            \leq2\norm[1]{f}.
        \]

        \item
        $
            \boff=\sum_kb_{\mathrm{off},k},
        $
        where
        \[
            b_{\mathrm{off},k}
            =p_k(f-f_k)q_k+q_k(f-f_k)p_k
            =p_kfq_k+q_kfp_k.
        \]
    \end{enumerate}
\end{theorem}

The cube decomposition of the diagonal bad part and the associated
localization projection will be used throughout the proof.

\begin{lemma}\label{lem:cz-localization}
    For $k\in\Z$ and $Q\in\Q_k$, set
    \[
        b_Q=\one_Q\pi_Q(f-f_Q)\pi_Q,
    \]
    and define
    \[
        \zeta
        =\left(
          \bigvee_k\bigvee_{Q\in\Q_k}
          \pi_Q\one_{2^{n_0+10}Q}
        \right)^\perp.
    \]
    Then the following assertions hold.
    \begin{enumerate}[label=\textup{(\roman*)},leftmargin=2.2em]
        \item For every $k\in\Z$,
        \[
            b_{\mathrm d,k}=\sum_{Q\in\Q_k}b_Q,
            \qquad
            \int_Qb_Q=0.
        \]
        Moreover,
        \[
            \sum_{Q\in\Q_k}\norm[1]{b_Q}
            =\norm[1]{b_{\mathrm d,k}}.
        \]

        \item The localization projection satisfies
        \begin{equation}\label{eq:localized-projection}
            \varphi(1-\zeta)
            \leq C_d\alpha^{-1}\norm[1]{f}.
        \end{equation}

        \item For every $Q\in\Q_k$,
        \[
            \zeta(x)\pi_Q=\pi_Q\zeta(x)=0,
            \qquad x\in2^{n_0+10}Q.
        \]
    \end{enumerate}
\end{lemma}

\begin{proof}
    The identities in (i) follow from the definitions and the
    disjointness of the cubes in $\Q_k$ in the Bochner space
    $L_1(\R^d;L_1(\M))$. Assertions (ii) and (iii) follow from
    \cite[Lemma~2.1]{CCP}; see also \cite[Theorem~2.5]{HLX}.
\end{proof}

We next record the localization consequence used below; compare
\cite[(3-3)]{Lai}.

\begin{lemma}\label{lemm:localsupport}
    Let $k,j\in\Z$ and $h_k\in L_1(\A)$ satisfy
    \[
        (1-p_k)h_k(1-p_k)=0.
    \]
    Let $J_j\in L_1(\R^d)$ be a scalar kernel such that
    \[
        \supp J_j\subset\{x:|x|\leq2^{j+1}\}.
    \]
    Then
    \begin{equation}\label{eq:localsupport}
        \zeta(J_j*h_k)\zeta=0
        \qquad(j+k<n_0).
    \end{equation}
\end{lemma}

\begin{proof}
    Since $(1-p_k)h_k(1-p_k)=0$,
    \[
        h_k=p_kh_k+(1-p_k)h_kp_k.
    \]
    Write
    \[
        p_k=\sum_{Q\in\Q_k}\pi_Q\one_Q.
    \]
    For $y\in Q\in\Q_k$,
    \[
        \begin{aligned}
        &\zeta(x)\int_QJ_j(x-y)h_k(y)\dd y\,\zeta(x)\\
        &\quad=
        \zeta(x)\pi_Q\int_QJ_j(x-y)h_k(y)\dd y\,\zeta(x)\\
        &\qquad+
        \zeta(x)\int_QJ_j(x-y)(1-\pi_Q)h_k(y)\dd y\,\pi_Q\zeta(x).
        \end{aligned}
    \]
    If $x\in2^{n_0+10}Q$, both terms vanish by
    Lemma~\ref{lem:cz-localization}(iii). If
    $x\notin2^{n_0+10}Q$, then for $y\in Q$,
    \[
        |x-y|>
        \bigl(2^{n_0+10}-1\bigr)2^{-k-1}
        >2^{j+1}
    \]
    whenever $j+k<n_0$. Hence $J_j(x-y)=0$, which proves
    \eqref{eq:localsupport}.
\end{proof}

\begin{corollary}\label{cor:badlocalvanishing}
    For $j,k\in\Z$ with $j+k<n_0$,
    \begin{equation}\label{eq:badlocalvanishing}
        \zeta(K_j^\Omega*b_{\mathrm d,k})\zeta
        =\zeta(K_j^\Omega*b_{\mathrm{off},k})\zeta
        =0.
    \end{equation}
\end{corollary}

\begin{proof}
    Since $\supp\beta\subset(1/2,2)$,
    \[
        \supp K_j^\Omega
        \subset\{x:2^{j-1}<|x|<2^{j+1}\}
        \subset\{x:|x|\leq2^{j+1}\}.
    \]
    Moreover,
    \[
        (1-p_k)b_{\mathrm d,k}(1-p_k)=0,
        \qquad
        (1-p_k)b_{\mathrm{off},k}(1-p_k)=0.
    \]
    The conclusion follows from Lemma~\ref{lemm:localsupport} with
    $J_j=K_j^\Omega$.
\end{proof}

\subsection{The functions
    \texorpdfstring{$b_{k,\ell}$ and $g_{k,\ell}$}{b and g}}\label{subsec-b-g}

For the analysis of the off-diagonal bad term $\boff$ in
Theorem~\ref{thm:regularcz}, we use two families introduced in
Parcet's original decomposition \cite{Parcet}.

\begin{lemma}\label{lem:lag-blocks}
	Let $f$ satisfy \eqref{eq:finitemodel}, and let $(p_k)_{k\in\Z}$
	and $(q_k)_{k\in\Z}$ be the projections defined above.
	For $k\in\Z$ and every integer $\ell\geq1$, define
    \begin{equation}\label{eq-bkl-gkl}
        \begin{aligned}
            b_{k,\ell}
            &=p_k(f-f_{k+\ell})p_{k+\ell}
              +p_{k+\ell}(f-f_{k+\ell})p_k,\\
            g_{k,\ell}
            &=p_kdf_{k+\ell}q_{k+\ell-1}
              +q_{k+\ell-1}df_{k+\ell}p_k.
        \end{aligned}
    \end{equation}
    Then the following assertions hold.
    \begin{enumerate}[label=\textup{(\roman*)},leftmargin=2.2em]
        \item The $b$-blocks satisfy
        \[
            \sum_k\norm[1]{b_{k,\ell}}
            \lesssim\norm[1]{f},
            \qquad
            \E_{k+\ell}b_{k,\ell}=0.
        \]

        \item The $g$-blocks satisfy
        \[
            \sum_k\norm[2]{g_{k,\ell}}^2
            \lesssim_d\alpha\norm[1]{f}.
        \]
        Moreover,
        \[
            g_{k,\ell}\in\A_{k+\ell},
            \qquad
            \E_{k+\ell-1}g_{k,\ell}=0.
        \]

        \item Under the  assumption \eqref{eq:finitemodel},
        \begin{equation}\label{eq:lagresolution}
            b_{\mathrm{off},k}
            =\sum_{\ell\geq1}(b_{k,\ell}+g_{k,\ell}),
            \qquad
            \boff
            =\sum_{\ell\geq1}\sum_k
            (b_{k,\ell}+g_{k,\ell}).
        \end{equation}
    \end{enumerate}
\end{lemma}

The estimates in Lemma~\ref{lem:lag-blocks}(i)--(ii) are stated in
\cite[Lemmas~4.2 and~3.7]{Lai}; their proofs are based on the
off-diagonal analysis in \cite{Parcet}; see also \cite{Cadilhac2018}
for the estimate of $b_{k,\ell}$.
In \cite{Parcet},  the family $b_{k, \ell}$ is used to decompose the off-diagonal bad function, whereas $g_{k, \ell}$ is used to decompose the off-diagonal good function; see
\cite[Sections~5.4 and~5.1]{Parcet}, respectively.

For assertion (iii), Remark~2.3 of \cite{CCP} identifies its new
off-diagonal bad term in the regular decomposition with the two
off-diagonal contributions in Parcet's decomposition. Combining this
with \cite[Sections~5.1 and~5.4]{Parcet} gives
\eqref{eq:lagresolution}. The   assumption \eqref{eq:finitemodel} ensures that only
finitely many terms in this identity are nonzero.

The resolution in Lemma~\ref{lem:lag-blocks}(iii) will be used in two
ways. For the terms involving $\Omega_{n+\ell}^\lo$, the $b$- and
$g$-blocks are estimated separately using (i) and (ii). For the terms
involving $\Omega_{n+\ell}^\hi$, they are kept together so that
Lemma~\ref{lem:cumulative} can be applied.

\section{Proof of the Main Theorem}\label{sec:plan}
	In this section, we prove Theorem~\ref{thm:main} assuming four
	auxiliary lemmas. The first is proved in
	Section~\ref{sec:diagproof}, while the remaining three are proved in
	Section~\ref{sec:offproof}.
	
	We first assume \eqref{eq:finitemodel} and establish the desired weak
	type estimate for the finite truncations
	\[
	T_{\Omega,I}f
	=
	\sum_{j\in I}K_j^\Omega*f,
	\]
	uniformly over finite sets $I\subset\Z$. The passage from
	$T_{\Omega,I}$ to $T_\Omega$ is carried out at the end of the section.
	
	The case $\Omega=0$ is immediate. Assume $\Omega\ne0$ and fix
	$\lambda>0$. Apply the regular noncommutative Calder\'on--Zygmund
	decomposition from Theorem~\ref{thm:regularcz} to $f$ at height
	$
	\alpha=\frac{\lambda}{C_{\Omega,1}},
	$
	and write
	\[
	f=g+\bd+\boff.
	\]
	By the distribution inequality,
	\[
	d_{7\lambda}(T_{\Omega,I}f)
	\leq
	d_\lambda(T_{\Omega,I}g)
	+d_{3\lambda}(T_{\Omega,I}\bd)
	+d_{3\lambda}(T_{\Omega,I}\boff).
	\]
Let $\zeta$ be the localization projection from
Lemma~\ref{lem:cz-localization}. For $h=\bd$ or $\boff$, set
	$H=T_{\Omega,I}h$. Since
	\[
	H=\zeta H\zeta+(1-\zeta)H+\zeta H(1-\zeta),
	\]
	the distribution inequality gives
	\begin{equation}\label{eq:bad-reduction}
		\begin{aligned}
			d_{3\lambda}(H)
			&\leq d_\lambda(\zeta H\zeta)
			+d_\lambda((1-\zeta)H)
			+d_\lambda(\zeta H(1-\zeta))\\
			&\leq d_\lambda(\zeta H\zeta)+2\varphi(1-\zeta).
		\end{aligned}
	\end{equation}
	Together with the estimate for $\varphi(1-\zeta)$ in
	\eqref{eq:localized-projection}, this reduces the proof to estimating
	$T_{\Omega,I}g$ and the two localized bad terms
	\[
	\zeta T_{\Omega,I}\bd\zeta,
	\qquad
	\zeta T_{\Omega,I}\boff\zeta.
	\]

	\subsection{The good function}
	Since $g\geq0$, functional calculus gives
	$g^2\leq\norm[\infty]{g}\,g$. Consequently, by
	Lemma~\ref{lem:fourier} and \eqref{eq:czbounds},
	\[
	\begin{aligned}
		\norm[2]{T_{\Omega,I}g}^2
		&\leq C_dC_{\Omega,1}^2\norm[2]{g}^2
		= C_dC_{\Omega,1}^2\varphi(g^2)\\
		&\leq C_dC_{\Omega,1}^2\norm[\infty]{g}\varphi(g)
		\leq C_dC_{\Omega,1}^2\alpha\norm[1]{f}\\
		&=C_dC_{\Omega,1}\lambda\norm[1]{f}.
	\end{aligned}
	\]
	Chebyshev's inequality therefore yields
	\begin{equation}\label{eq:goodweak}
		d_\lambda(T_{\Omega,I}g)
		\leq\lambda^{-2}\norm[2]{T_{\Omega,I}g}^2
		\leq C_d C_{\Omega,1}\lambda^{-1}\norm[1]{f}.
	\end{equation}
	
	\subsection{The diagonal bad function}
	Using \eqref{eq:badlocalvanishing}, we have
	\begin{equation}\label{eq:diaglocal}
		\begin{split}
			\zeta T_{\Omega,I}\bd\zeta
			&=\zeta\sum_{j\in I}\sum_kK_j^\Omega*b_{\mathrm d,k}\zeta\\
			&=\zeta\sum_{s\geq n_0}\sum_{j\in I}K_j^\Omega*b_{\mathrm d,s-j}\zeta\\
			&=\zeta\sum_{s\geq n_0}\sum_{j\in I}
			[\Gamma_j^{s,\Omega_s^\lo}+(K_j^{\Omega_s^\lo}-\Gamma_j^{s,\Omega_s^\lo})+K_j^{\Omega_s^\hi}]*b_{\mathrm d,s-j}\zeta,
		\end{split}
	\end{equation}
	where the second line follows by setting $s=j+k$, and the terms with
	$s<n_0$ vanish by \eqref{eq:badlocalvanishing},
	the third follows from
	\eqref{eq:amplitudekernels}.   We estimate the three sums in the last line separately.
	
	\subsubsection{The sum with \texorpdfstring{$\Gamma_j^{s,\Omega_s^\lo}$}{Gamma}}
	
	For this sum we need an $L_2$ estimate, uniform in the finite
	set $I$, which is stated below and   proved in Section \ref{sec:diagproof}.
	
	\begin{lemma}\label{lem:diagprincipal}
		For  every integer $s\geq n_0$ and $\Psi\in L_\infty(S^{d-1})$,
		\[
		\left\|\sum_{j\in I}\Gamma_j^{s,\Psi}*b_{\mathrm d,s-j}\right\|_2^2
		\leq C_d\norm[\infty]{\Psi}^2\alpha\norm[1]{f}\,2^{-s/2}.
		\]
	\end{lemma}

	Take
	$\Psi=\Omega_s^\lo$, use
	$\norm[\infty]{\Omega_s^\lo}\leq2^{\delta s}\norm[1]{\Omega}$,
	and apply the triangle inequality in $L_2(\A)$. We obtain
	\[
	\begin{aligned}
		\left\|\sum_{s\geq n_0}\sum_{j\in I}\Gamma_j^{s,\Omega_s^\lo}*b_{\mathrm d,s-j}\right\|_2
		&\leq\sum_{s\geq n_0}
		\left\|\sum_{j\in I}\Gamma_j^{s,\Omega_s^\lo}*b_{\mathrm d,s-j}\right\|_2\\
		&\leq C_d\norm[1]{\Omega}
		(\alpha\norm[1]{f})^{1/2}
		\sum_{s\geq n_0}2^{-s/4+\delta s}\\
	 		&\leq C_d\norm[1]{\Omega}(\alpha\norm[1]{f})^{1/2}.
	\end{aligned}
	\]
	Here we used $\delta=1/16$.
	
	\subsubsection{The sum with \texorpdfstring{$K_j^{\Omega_s^\lo}-\Gamma_j^{s,\Omega_s^\lo}$}{K minus Gamma}}
	
	For $Q\in\Q_{s-j}$, the function $b_Q$ has mean zero and is supported in
	a cube of side $2^{j-s}$. Lemma~\ref{lem:remainder} gives
	\[
	\norm[1]{(K_j^{\Omega_s^\lo}-\Gamma_j^{s,\Omega_s^\lo})*b_Q}
	\leq C_d\norm[1]{\Omega}\,2^{-3s/16}\norm[1]{b_Q}.
	\]
	Since cubes at one generation are disjoint,
	\[
	\sum_{j\in I}\sum_{Q\in\Q_{s-j}}\norm[1]{b_Q}
	=\sum_{j\in I}\norm[1]{b_{\mathrm d,s-j}}
	\leq\sum_{k\in\Z}\norm[1]{b_{\mathrm d,k}}
	\leq2\norm[1]{f}.
	\]
	It follows that
	\[
	\begin{aligned}
		&\left\|\sum_{s\geq n_0}\sum_{j\in I}
		(K_j^{\Omega_s^\lo}-\Gamma_j^{s,\Omega_s^\lo})*b_{\mathrm d,s-j}\right\|_1\\
		&\leq\sum_{s\geq n_0}\sum_{j\in I}
		\sum_{Q\in\Q_{s-j}}\norm[1]{(K_j^{\Omega_s^\lo}-\Gamma_j^{s,\Omega_s^\lo})*b_Q}\\
		&\leq C_d\norm[1]{\Omega}
		\sum_{s\geq n_0}2^{-3s/16}
		\sum_{j\in I}\sum_{Q\in\Q_{s-j}}\norm[1]{b_Q}\\
		&\leq C_d\norm[1]{\Omega}\norm[1]{f}.
	\end{aligned}
	\]
	
	\subsubsection{The sum with \texorpdfstring{$K_j^{\Omega_s^\hi}$}{K}}
	
	Young's inequality \eqref{eq-Young}, \eqref{eq:kernelL1} and Theorem~\ref{thm:regularcz} give
	\[
	\begin{aligned}
		\left\|\sum_{s\geq n_0}\sum_{j\in I}K_j^{\Omega_s^\hi}*b_{\mathrm d,s-j}\right\|_1
		&\leq\sum_{s\geq n_0}\sum_{j\in I}
		\norm[1]{K_j^{\Omega_s^\hi}}\norm[1]{b_{\mathrm d,s-j}}\\
		&\leq C_d\sum_{s\geq n_0}\norm[1]{\Omega_s^\hi}
		\sum_{j\in I}\norm[1]{b_{\mathrm d,s-j}}\\
		&\leq C_d\norm[1]{f}\sum_{s\geq n_0}
		\int_{\{|\Omega|>2^{\delta s}\norm[1]{\Omega}\}}
		|\Omega(\theta)|\dd\sigma(\theta)\\
		&\leq C_d C_{\Omega,1}\norm[1]{f},
	\end{aligned}
	\]
 In the last inequality, we use $
  \sum_{s\geq n_0}
  \int_{\{|\Omega|>2^{\delta s}\|\Omega\|_1\}}
  |\Omega(\theta)|\dd\sigma(\theta)
  \lesssim C_{\Omega,1}.
 $

	\subsubsection{Completion of the diagonal estimate}

	Using \eqref{eq:bad-reduction} and \eqref{eq:diaglocal}, we split $\zeta T_{\Omega,I}\bd\zeta$ into its three
	terms and apply the Chebyshev inequality with threshold
	$\lambda/3$ to each term.  We obtain
	\begin{equation}\label{eq:diagweak}
		\begin{split}
			d_{3\lambda}(T_{\Omega,I}\bd)
			&\leq2\varphi(1-\zeta)
			+\frac9{\lambda^2}
			\left\|\sum_{s\geq n_0}\sum_{j\in I}\Gamma_j^{s,\Omega_s^\lo}*b_{\mathrm d,s-j}\right\|_2^2\\
			&\quad+\frac3\lambda
			\left\|\sum_{s\geq n_0}\sum_{j\in I}(K_j^{\Omega_s^\lo}-\Gamma_j^{s,\Omega_s^\lo})*b_{\mathrm d,s-j}\right\|_1\\
			&\quad+\frac3\lambda
			\left\|\sum_{s\geq n_0}\sum_{j\in I}K_j^{\Omega_s^\hi}*b_{\mathrm d,s-j}\right\|_1\\
			&\leq C_d\left[\alpha^{-1}
			+\lambda^{-2}\norm[1]{\Omega}^2\alpha
			+\lambda^{-1}C_{\Omega,1}\right]\norm[1]{f}\\
			&\leq C_d C_{\Omega,1}\lambda^{-1}\norm[1]{f}.
		\end{split}
	\end{equation}

	\subsection{The off-diagonal bad function}
By \eqref{eq:badlocalvanishing} and setting $n=j+k$, we have
	\[
	\begin{aligned}
		\zeta T_{\Omega,I}\boff\zeta
		&=\zeta\sum_{n\geq n_0}\sum_{j\in I}
		K_j^\Omega*b_{\mathrm{off},n-j}\zeta\\
		&=\zeta\sum_{n\geq n_0}\sum_{\ell\geq1}\sum_{j\in I}
		K_j^\Omega*(b_{n-j,\ell}+g_{n-j,\ell})\zeta,
	\end{aligned}
	\]
	where the second equality follows from \eqref{eq:lagresolution}.
	Apply  \eqref{eq:amplitudekernels} with $s=n+\ell$.  On the $b$-block we use all three terms in
	\eqref{eq:amplitudekernels}. On the $g$-block we only use
	$K_j^\Omega=K_j^{\Omega_{n+\ell}^\lo}+K_j^{\Omega_{n+\ell}^\hi}$. Consequently,
	\begin{equation}\label{eq:offsplit}
		\begin{split}
			\zeta T_{\Omega,I}\boff\zeta
			&=\zeta\sum_{n\geq n_0}\sum_{\ell\geq1}\sum_{j\in I}
			K_j^{\Omega_{n+\ell}^\hi}*(b_{n-j,\ell}+g_{n-j,\ell})\zeta\\
			&\quad+\zeta\sum_{n\geq n_0}\sum_{\ell\geq1}\sum_{j\in I}
			K_j^{\Omega_{n+\ell}^\lo}*g_{n-j,\ell}\zeta\\
			&\quad+\zeta\sum_{n\geq n_0}\sum_{\ell\geq1}\sum_{j\in I}
			\Gamma_j^{n+\ell,\Omega_{n+\ell}^\lo}*b_{n-j,\ell}\zeta\\
			&\quad+\zeta\sum_{n\geq n_0}\sum_{\ell\geq1}\sum_{j\in I}
			(K_j^{\Omega_{n+\ell}^\lo}-\Gamma_j^{n+\ell,\Omega_{n+\ell}^\lo})*b_{n-j,\ell}\zeta.
		\end{split}
	\end{equation}
	Below we estimate these four terms separately.
	
\subsubsection{The sum with
	\texorpdfstring{$K_j^{\Omega_{n+\ell}^\hi}$}{K}}

For fixed $\theta$, the condition in the definition of
$\Omega_{n+\ell}^{\hi}$ depends on $n$ and $\ell$ only through
$n+\ell$. This leads to the following cumulative $L_1$ estimate,
proved in Section~\ref{sec:cumulativeproof}.

\begin{lemma}\label{lem:cumulative}
	For every integer $h\geq n_0$ and $j\in\Z$,
	\begin{equation}\label{eq:cumulativebounds}
		\left\|
		\sum_{\substack{n\geq n_0,\ \ell\geq1\\ n+\ell\leq h}}
		(b_{n-j,\ell}+g_{n-j,\ell})
		\right\|_1
		\leq
		4\sum_{k=n_0-j}^{h-j}\varphi(fp_k).
	\end{equation}
	Consequently,
	\begin{equation}\label{eq:cumulativeglobal}
		\sum_{j\in\Z}
		\left\|
		\sum_{\substack{n\geq n_0,\ \ell\geq1\\ n+\ell\leq h}}
		(b_{n-j,\ell}+g_{n-j,\ell})
		\right\|_1
		\leq
		4(h-n_0+1)\norm[1]{f}.
	\end{equation}
\end{lemma}

Using the definition of $K_j^{\Omega_{n+\ell}^\hi}$ and polar
coordinates, we obtain
\begin{equation}\label{eq:offhirearrange}
	\begin{split}
		&\sum_{n\geq n_0}\sum_{\ell\geq1}\sum_{j\in I}
		K_j^{\Omega_{n+\ell}^\hi}
		*(b_{n-j,\ell}+g_{n-j,\ell})(x)\\
		&\quad=
		\int_{S^{d-1}}\Omega(\theta)
		\sum_{j\in I}\int_0^\infty
		\beta(2^{-j}r)^2\\
		&\qquad\qquad\times
		\sum_{\substack{n\geq n_0,\ \ell\geq1\\
				2^{\delta(n+\ell)}\norm[1]{\Omega}
				<|\Omega(\theta)|}}
		(b_{n-j,\ell}+g_{n-j,\ell})(x-r\theta)
		\frac{\dd r}{r}\dd\sigma(\theta).
	\end{split}
\end{equation}

Fix $\theta$ for which the index set of the inner sum is nonempty,
and set
\[
h(\theta)
=
\left\lceil
\delta^{-1}\log_2
\frac{|\Omega(\theta)|}{\norm[1]{\Omega}}
\right\rceil-1.
\]
Then $h(\theta)\geq n_0+1$, and, since $n+\ell$ is an integer,
\[
2^{\delta(n+\ell)}\norm[1]{\Omega}<|\Omega(\theta)|
\quad\Longleftrightarrow\quad
n+\ell\leq h(\theta).
\]
Hence, by \eqref{eq:cumulativeglobal},
\begin{equation}\label{eq:angularcumulative}
	\begin{aligned}
		&\sum_{j\in I}
		\left\|
		\sum_{\substack{n\geq n_0,\ \ell\geq1\\
				2^{\delta(n+\ell)}\norm[1]{\Omega}
				<|\Omega(\theta)|}}
		(b_{n-j,\ell}+g_{n-j,\ell})
		\right\|_1\\
		&\qquad\leq
		4\bigl(h(\theta)-n_0+1\bigr)\norm[1]{f}\\
		&\qquad\lesssim
		\left(
		1+\log^+
		\frac{|\Omega(\theta)|}{\norm[1]{\Omega}}
		\right)\norm[1]{f}.
	\end{aligned}
\end{equation}
Here we used $\delta=1/16$. The final bound also holds when the
index set is empty, since the corresponding sum is zero.

Finally, Minkowski's inequality, \eqref{eq:offhirearrange}, and
\eqref{eq:angularcumulative} give
\begin{equation}\label{eq:offhi}
	\begin{aligned}
		&\left\|
		\sum_{n\geq n_0}\sum_{\ell\geq1}\sum_{j\in I}
		K_j^{\Omega_{n+\ell}^\hi}
		*(b_{n-j,\ell}+g_{n-j,\ell})
		\right\|_1\\
		&\quad\lesssim
		\int_{S^{d-1}}|\Omega(\theta)|
		\sum_{j\in I}
		\left\|
		\sum_{\substack{n\geq n_0,\ \ell\geq1\\
				2^{\delta(n+\ell)}\norm[1]{\Omega}
				<|\Omega(\theta)|}}
		(b_{n-j,\ell}+g_{n-j,\ell})
		\right\|_1
		\dd\sigma(\theta)\\
		&\quad\lesssim
		\norm[1]{f}
		\int_{S^{d-1}}|\Omega(\theta)|
		\left(
		1+\log^+
		\frac{|\Omega(\theta)|}{\norm[1]{\Omega}}
		\right)\dd\sigma(\theta)\\
		&\quad =
		 C_{\Omega,1}\norm[1]{f}.
	\end{aligned}
\end{equation}

	\subsubsection{The sum with
		\texorpdfstring{$K_j^{\Omega_{n+\ell}^\lo}*g_{n-j,\ell}$}{K times g}}
	
	Fix  integers $n\geq n_0$ and $\ell\geq1$, and set $s=n+\ell$.
	By Lemma~\ref{lem:lag-blocks}(ii),
	$\E_{s-j-1}g_{n-j,\ell}=0$. We therefore need an estimate that
	combines this martingale cancellation with the Fourier decay of
	$K_j^\Psi$. The required estimate is the following; its proof is
	given in Section~\ref{sec:martingaleproof}.
	
	\begin{lemma}\label{lem:martingale}
		Let $s\geq1$ be an integer and let $(h_j)_{j\in I}\subset L_2(\A)$ satisfy
		$\E_{s-j-1}h_j=0$. For $\Psi\in L_\infty(S^{d-1})$,
		\begin{equation}\label{eq:martingalesynthesis}
			\left\|\sum_{j\in I}K_j^\Psi*h_j\right\|_2
			\leq C_d\norm[\infty]{\Psi}(1+s)2^{-s}
			\left(\sum_{j\in I}\norm[2]{h_j}^2\right)^{1/2}.
		\end{equation}
	\end{lemma}
	Apply Lemma~\ref{lem:martingale} with
	$h_j=g_{n-j,\ell}$ and $\Psi=\Omega_s^\lo$.
	By Lemma~\ref{lem:lag-blocks}(ii) and
	$\|\Omega_s^\lo\|_\infty\leq2^{\delta s}\|\Omega\|_1$,
	we obtain
  	\[
	\left\|\sum_{j\in I}
	K_j^{\Omega_{n+\ell}^\lo}*g_{n-j,\ell}\right\|_2
	\leq
	C_d\norm[1]{\Omega}(1+n+\ell)
	2^{-(1-\delta)(n+\ell)}
	(\alpha\norm[1]{f})^{1/2}.
	\]
	Therefore, by the triangle inequality,
	\begin{equation}\label{eq:glosum}
		\begin{split}
			&\left\|\sum_{n\geq n_0}\sum_{\ell\geq1}\sum_{j\in I}
			K_j^{\Omega_{n+\ell}^\lo}*g_{n-j,\ell}\right\|_2\\
			&\quad\leq
			C_d\norm[1]{\Omega}(\alpha\norm[1]{f})^{1/2}
			\sum_{n\geq n_0}\sum_{\ell\geq1}
			(1+n+\ell)2^{-15(n+\ell)/16}\\
			&\quad\leq
			C_d\norm[1]{\Omega}(\alpha\norm[1]{f})^{1/2}.
		\end{split}
	\end{equation}

	\subsubsection {The sums with \texorpdfstring{$b_{n-j,\ell}$}{b}}
	It remains to estimate the two sums involving $b_{n-j,\ell}$ in
	\eqref{eq:offsplit}. The term with
	$K_j^{\Omega_{n+\ell}^\lo}
	-\Gamma_j^{n+\ell,\Omega_{n+\ell}^\lo}$
	is controlled by Lemma~\ref{lem:remainder}. For the term with
	$\Gamma_j^{n+\ell,\Omega_{n+\ell}^\lo}$, we need the following
	$L_2$ estimate. Its main difficulty is the control of the terms with
	$i\ne j$ arising from the expansion of the squared $L_2$ norm. In
	the randomized setting of \cite[Lemma~3.4]{Lai}, orthogonality in
	the auxiliary probability variable eliminates these cross terms. The proof is given in Section \ref{sec:twosided}.
	
	\begin{lemma}\label{lem:offprincipal}
	 For $n\geq n_0$, $\ell\geq1$, $s=n+\ell$, and $\Psi\in L_\infty(S^{d-1})$,
		\begin{equation}\label{eq:offprincipal}
			\left\|\sum_{j\in I}\Gamma_j^{s,\Psi}*b_{n-j,\ell}\right\|_2^2
			\leq C_d\norm[\infty]{\Psi}^2(1+\ell)\alpha\norm[1]{f}\,2^{-s/2}.
		\end{equation}
	\end{lemma}
	
	Taking $\Psi=\Omega_{n+\ell}^\lo$ in Lemma~\ref{lem:offprincipal}, we obtain
	\[
	\left\|\sum_{j\in I}
	\Gamma_j^{n+\ell,\Omega_{n+\ell}^\lo}*b_{n-j,\ell}\right\|_2
	\leq
	C_d\norm[1]{\Omega}(1+\ell)^{1/2}
	2^{-3(n+\ell)/16}
	(\alpha\norm[1]{f})^{1/2}.
	\]
	Therefore, by the triangle inequality,
	\begin{equation}\label{eq:offprsum}
		\begin{split}
			&\left\|
			\sum_{n\geq n_0}\sum_{\ell\geq1}\sum_{j\in I}
			\Gamma_j^{n+\ell,\Omega_{n+\ell}^\lo}*b_{n-j,\ell}
			\right\|_2\\
			&\quad\leq
			C_d\norm[1]{\Omega}(\alpha\norm[1]{f})^{1/2}
			\sum_{n\geq n_0}\sum_{\ell\geq1}
			(1+\ell)^{1/2}2^{-3(n+\ell)/16}\\
			&\quad\leq
			C_d\norm[1]{\Omega}(\alpha\norm[1]{f})^{1/2}.
		\end{split}
	\end{equation}
	For the sum with $K_j^{\Omega_{n+\ell}^\lo}-\Gamma_j^{n+\ell,\Omega_{n+\ell}^\lo}$, write
	\[
	b_{n-j,\ell}=\sum_{Q\in\Q_{n+\ell-j}}\one_Qb_{n-j,\ell},\qquad
	\int_Q\one_Qb_{n-j,\ell}=0.
	\]
	Each cube has side $2^{j-(n+\ell)}$. Thus Lemma~\ref{lem:remainder}
	and Lemma~\ref{lem:lag-blocks}(i) give
	\begin{equation}\label{eq:offrem}
		\begin{split}
			&\left\|\sum_{n\geq n_0}\sum_{\ell\geq1}\sum_{j\in I}
			(K_j^{\Omega_{n+\ell}^\lo}-\Gamma_j^{n+\ell,\Omega_{n+\ell}^\lo})*b_{n-j,\ell}\right\|_1\\
			&\quad\leq C_d\norm[1]{\Omega}
			\sum_{n\geq n_0}\sum_{\ell\geq1}2^{-3(n+\ell)/16}
			\sum_{j\in I}\sum_{Q\in\Q_{n+\ell-j}}
			\norm[1]{\one_Qb_{n-j,\ell}}\\
			&\quad=C_d\norm[1]{\Omega}
			\sum_{n\geq n_0}\sum_{\ell\geq1}2^{-3(n+\ell)/16}
			\sum_{j\in I}\norm[1]{b_{n-j,\ell}}\\
			&\quad\leq C_d\norm[1]{\Omega}\norm[1]{f}
			\sum_{n\geq n_0}\sum_{\ell\geq1}2^{-3(n+\ell)/16}\\
			&\quad\leq C_d\norm[1]{\Omega}\norm[1]{f}.
		\end{split}
	\end{equation}
	
	\subsubsection{Completion of the off-diagonal estimate}
	Using \eqref{eq:bad-reduction}, we split
	\eqref{eq:offsplit} into its four terms and apply Chebyshev's
	inequality with threshold $\lambda/4$ to each term.
	Equations \eqref{eq:offhi}, \eqref{eq:glosum}, \eqref{eq:offprsum},
	and \eqref{eq:offrem} yield
	\begin{equation}\label{eq:offweak}
		\begin{split}
			d_{3\lambda}(T_{\Omega,I}\boff)
			&\leq2\varphi(1-\zeta)
			+\frac{16}{\lambda^2}
			\left\|\sum_{n\geq n_0}\sum_{\ell\geq1}\sum_{j\in I}
			K_j^{\Omega_{n+\ell}^\lo}*g_{n-j,\ell}\right\|_2^2\\
			&\quad+\frac{16}{\lambda^2}
			\left\|\sum_{n\geq n_0}\sum_{\ell\geq1}\sum_{j\in I}
			\Gamma_j^{n+\ell,\Omega_{n+\ell}^\lo}*b_{n-j,\ell}\right\|_2^2\\
			&\quad+\frac4\lambda
			\left\|\sum_{n\geq n_0}\sum_{\ell\geq1}\sum_{j\in I}
			K_j^{\Omega_{n+\ell}^\hi}*(b_{n-j,\ell}+g_{n-j,\ell})\right\|_1\\
			&\quad+\frac4\lambda
			\left\|\sum_{n\geq n_0}\sum_{\ell\geq1}\sum_{j\in I}
			(K_j^{\Omega_{n+\ell}^\lo}-\Gamma_j^{n+\ell,\Omega_{n+\ell}^\lo})*b_{n-j,\ell}\right\|_1\\
			&\leq C_d\left[\alpha^{-1}
			+\lambda^{-2}\norm[1]{\Omega}^2\alpha
			+\lambda^{-1}C_{\Omega,1}\right]\norm[1]{f}\\
			&\leq C_d C_{\Omega,1}\lambda^{-1}\norm[1]{f}.
		\end{split}
	\end{equation}

\subsection{Completion of the proof of Theorem~\ref{thm:main}}
\label{sec:completion}

\noindent\textit{Step 1.}
Let $f$ satisfy \eqref{eq:finitemodel}. By
\eqref{eq:goodweak}, \eqref{eq:diagweak}, and
\eqref{eq:offweak},
\begin{equation}\label{eq:finiteweak}
	\begin{aligned}
		d_{7\lambda}(T_{\Omega,I}f)
		&\leq d_\lambda(T_{\Omega,I}g)
		+d_{3\lambda}(T_{\Omega,I}\bd)
		+d_{3\lambda}(T_{\Omega,I}\boff)\\
		&\leq
		C_dC_{\Omega,1}\lambda^{-1}\norm[1]{f}.
	\end{aligned}
\end{equation}
By \eqref{eq:L2-convergence},
\[
T_{\Omega,[-J,J]}f\longrightarrow T_\Omega f
\qquad\text{in }L_2(\A).
\]
Hence, applying \eqref{eq:finiteweak} with $\lambda=t/14$,
\[
\begin{aligned}
	d_t(T_\Omega f)
	&\leq
	d_{t/2}(T_{\Omega,[-J,J]}f)
	+4t^{-2}
	\norm[2]{T_\Omega f-T_{\Omega,[-J,J]}f}^2\\
	&\leq
	C_dC_{\Omega,1}t^{-1}\norm[1]{f}
	+4t^{-2}
	\norm[2]{T_\Omega f-T_{\Omega,[-J,J]}f}^2.
\end{aligned}
\]
Letting $J\to\infty$ gives
\begin{equation}\label{eq:finiteOmegaweak}
	d_t(T_\Omega f)
	\leq
	C_dC_{\Omega,1}t^{-1}\norm[1]{f}.
\end{equation}

\medskip
\noindent\textit{Step 2.}
Let $f\in L_1(\A)_+\cap\A$ be compactly supported and set
$f_N=\E_Nf$. Then $f_N$ satisfies \eqref{eq:finitemodel} with
$M=N$, and
\[
\norm[1]{f_N}=\norm[1]{f},
\qquad
f_N\longrightarrow f
\quad\text{in }L_2(\A).
\]
Since $T_\Omega$ is bounded on $L_2(\A)$,
\[
T_\Omega f_N\longrightarrow T_\Omega f
\qquad\text{in }L_2(\A).
\]
By \eqref{eq:finiteOmegaweak},
\[
\begin{aligned}
	d_t(T_\Omega f)
	&\leq
	d_{t/2}(T_\Omega f_N)
	+4t^{-2}\norm[2]{T_\Omega f-T_\Omega f_N}^2\\
	&\leq
	C_dC_{\Omega,1}t^{-1}\norm[1]{f}
	+4t^{-2}\norm[2]{T_\Omega f-T_\Omega f_N}^2.
\end{aligned}
\]
Letting $N\to\infty$ gives
\[
d_t(T_\Omega f)
\leq
C_dC_{\Omega,1}t^{-1}\norm[1]{f}.
\]

\medskip
\noindent\textit{Step 3.}
Let $f\in L_1(\A)$ be bounded and compactly supported. Write
\[
f=(\operatorname{Re}f)_+-(\operatorname{Re}f)_-
+i(\operatorname{Im}f)_+-i(\operatorname{Im}f)_-.
\]
The four positive parts
$(\operatorname{Re}f)_\pm$ and $(\operatorname{Im}f)_\pm$
are bounded and compactly supported, and
\[
\|(\operatorname{Re}f)_+\|_1+\|(\operatorname{Re}f)_-\|_1
+\|(\operatorname{Im}f)_+\|_1+\|(\operatorname{Im}f)_-\|_1
\leq2\|f\|_1.
\]
Applying the preceding estimate to these four terms and using the
distribution inequality, we obtain
\[
\norm[1,\infty]{T_\Omega f}
\leq
C_dC_{\Omega,1}\norm[1]{f}.
\]
By density and the completeness of $L_{1,\infty}(\A)$,
$T_\Omega$ extends uniquely to a bounded linear map
\[
T_\Omega:L_1(\A)\longrightarrow L_{1,\infty}(\A).
\]
 This proves Theorem~\ref{thm:main}.

	\section{Proof of the diagonal lemma}\label{sec:diagproof}

We prove Lemma~\ref{lem:diagprincipal}. We first isolate the scalar
geometric estimates that will be used here and again in the proof of
Lemma~\ref{lem:offprincipal}. For complex-valued $\Psi$, set
\[
\widetilde H_{j\nu}^{s,\Psi}(x)
=
\overline{H_{j\nu}^{s,\Psi}(-x)}.
\]
The proof of \cite[Lemma~2.1]{Seeger} gives the following standard
sector estimates. The size of $E_\nu^s$ and polar coordinates give
\begin{equation}\label{eq:sectornorms}
\|H_{j\nu}^{s,\Psi}\|_\infty
\lesssim_d \|\Psi\|_\infty 2^{-jd},
\qquad
\|H_{j\nu}^{s,\Psi}\|_1
\lesssim_d \|\Psi\|_\infty 2^{-s(d-1)}.
\end{equation}
For $i\leq j$, there is a centered rectangle $R_{j\nu}^s$, with one
side comparable to $2^j$ in the direction $e_\nu^s$ and the remaining
sides comparable to $2^{j-s}$, such that
\begin{equation}\label{eq:rectangle}
\supp(\widetilde H_{j\nu}^{s,\Psi}*H_{i\nu}^{s,\Psi})
\subset R_{j\nu}^s,
\qquad
|R_{j\nu}^s|\lesssim_d2^{jd}2^{-s(d-1)},
\end{equation}
and
\begin{equation}\label{eq:interactionmajorant}
|(\widetilde H_{j\nu}^{s,\Psi}*H_{i\nu}^{s,\Psi})(z)|
\lesssim_d
\|\Psi\|_\infty^2
2^{-jd}2^{-s(d-1)}\one_{R_{j\nu}^s}(z),
\qquad i\leq j.
\end{equation}
Moreover, if a dyadic cube $Q$ with $\ell(Q)\leq2^{j-s}$ meets
$x-R_{j\nu}^s$, then
\begin{equation}\label{eq:cubecover}
Q\subset x-C_dR_{j\nu}^s,
\end{equation}
where $C_dR_{j\nu}^s$ is a fixed dimensional dilation of
$R_{j\nu}^s$. In particular,
\begin{equation}\label{eq:finecubecover}
\sum_{\substack{Q\in\Q_{s-j}\\
Q\cap\supp\bigl(H_{j\nu}^{s,\Psi}(x-\cdot)\bigr)\neq\varnothing}}
|Q|
\lesssim_d 2^{jd}2^{-s(d-1)}.
\end{equation}

We shall repeatedly use the following elementary packing observation.
\begin{lemma}\label{lem:projection-packing}
Let $E\subset\R^d$ have finite measure, and let $\mathscr C$ be a
finite collection of dyadic cubes contained in $E$. Then
\begin{equation}\label{eq:projectionpacking}
\sum_{Q\in\mathscr C}|Q|\pi_Q
\leq |E|\one_\M.
\end{equation}
\end{lemma}

\begin{proof}
Using $p_k=\sum_{Q\in\Q_k}\pi_Q\one_Q$ and the orthogonality of the
$p_k$'s, we have
\[
\begin{aligned}
\sum_{Q\in\mathscr C}|Q|\pi_Q
&=
\int_E\sum_k\sum_{Q\in\mathscr C\cap\Q_k}
\one_Q(y)\pi_Q\dd y\\
&\leq
\int_E\sum_kp_k(y)\dd y
=
\int_E(1-q(y))\dd y
\leq |E|\one_\M.
\end{aligned}
\]
\end{proof}

The next estimate packages the geometric argument needed both in the
present proof and in the far-interaction part of
Lemma~\ref{lem:offprincipal}.
\begin{lemma}\label{lem:sector-packing}
Assume that $f$ satisfies \eqref{eq:finitemodel}. Fix
$s\geq n_0$, an index $\nu$, and $\Psi\in L_\infty(S^{d-1})$.
If $r\leq s$, then for every $j\in I$ and $x\in\R^d$,
\begin{equation}\label{eq:sector-packing}
\begin{aligned}
&\sum_{\substack{i\in I\\ i\leq j-(s-r)}}
\int_{\R^d}
|(\widetilde H_{j\nu}^{s,\Psi}*H_{i\nu}^{s,\Psi})(x-y)|
\,p_{r-i}(y)f(y)p_{r-i}(y)\dd y\\
&\qquad\lesssim_d
\alpha\|\Psi\|_\infty^2
2^{-2s(d-1)}\one_\M.
\end{aligned}
\end{equation}
The same bound holds with $f$ replaced by $f_{s-i}$.
\end{lemma}

\begin{proof}
For $Q\in\Q_{r-i}$, the cube estimate \eqref{eq:cubeproperties} gives
\begin{equation}\label{eq:sector-cube-average}
\int_Qp_{r-i}(y)f(y)p_{r-i}(y)\dd y
=
|Q|\pi_Qf_Q\pi_Q
\leq 2^d\alpha|Q|\pi_Q.
\end{equation}
Since $r\leq s$ and $\E_{r-i}f_{s-i}=f_{r-i}$, the same identity and
bound hold with $f$ replaced by $f_{s-i}$.

If $i\leq j-(s-r)$, then every $Q\in\Q_{r-i}$ has
\[
\ell(Q)=2^{i-r}\leq2^{j-s}.
\]
Thus, by \eqref{eq:cubecover}, every such cube meeting
$x-R_{j\nu}^s$ is contained in $x-C_dR_{j\nu}^s$. Using
\eqref{eq:interactionmajorant}, \eqref{eq:sector-cube-average},
Lemma~\ref{lem:projection-packing}, and \eqref{eq:rectangle}, we obtain
\[
\begin{aligned}
&\sum_{\substack{i\in I\\ i\leq j-(s-r)}}
\int_{\R^d}
|(\widetilde H_{j\nu}^{s,\Psi}*H_{i\nu}^{s,\Psi})(x-y)|
\,p_{r-i}(y)f(y)p_{r-i}(y)\dd y\\
&\quad\lesssim_d
\|\Psi\|_\infty^2 2^{-jd}2^{-s(d-1)}
\sum_{\substack{i\in I,\ i\leq j-(s-r)\\
Q\in\Q_{r-i},\ Q\cap(x-R_{j\nu}^s)\ne\varnothing}}
\int_Qp_{r-i}(y)f(y)p_{r-i}(y)\dd y\\
&\quad\lesssim_d
\alpha\|\Psi\|_\infty^2 2^{-jd}2^{-s(d-1)}
\sum_{\substack{i\in I,\ i\leq j-(s-r)\\
Q\in\Q_{r-i},\ Q\subset x-C_dR_{j\nu}^s}}
|Q|\pi_Q\\
&\quad\lesssim_d
\alpha\|\Psi\|_\infty^2 2^{-jd}2^{-s(d-1)}
|C_dR_{j\nu}^s|\one_\M\\
&\quad\lesssim_d
\alpha\|\Psi\|_\infty^2
2^{-2s(d-1)}\one_\M.
\end{aligned}
\]
The proof for $f_{s-i}$ is identical, using the second assertion in
\eqref{eq:sector-cube-average}.
\end{proof}

For later use, we also record the standard convolution pairing
identity. By Fubini's theorem and the traciality of $\tau$, for
$i\leq j$ and $h_i,h_j\in L_2(\A)$,
\begin{equation}\label{eq:convolutionpairing}
\begin{aligned}
&\left\langle
H_{i\nu}^{s,\Psi}*h_i,
H_{j\nu}^{s,\Psi}*h_j
\right\rangle_2\\
&\qquad=
\iint_{\R^d\times\R^d}
(\widetilde H_{j\nu}^{s,\Psi}*H_{i\nu}^{s,\Psi})(x-y)
\tau\bigl(h_i(y)h_j(x)^*\bigr)\dd y\dd x.
\end{aligned}
\end{equation}

\begin{proof}[Proof of Lemma~\ref{lem:diagprincipal}]
By \eqref{eq:Seegerkernels} and Lemma~\ref{lem:overlap},
\begin{equation}\label{eq:diag-overlap-reduction}
\begin{aligned}
\left\|
\sum_{j\in I}\Gamma_j^{s,\Psi}*b_{\mathrm d,s-j}
\right\|_2^2
&\leq
C_d2^{s(d-3/2)}
\sum_\nu
\left\|
\sum_{j\in I}H_{j\nu}^{s,\Psi}*b_{\mathrm d,s-j}
\right\|_2^2.
\end{aligned}
\end{equation}
Since \eqref{eq:angularpartition} gives
$\#\{\nu\}\lesssim_d2^{s(d-1)}$, it is enough to prove, uniformly in
$\nu$,
\begin{equation}\label{eq:diagsector}
\left\|
\sum_{j\in I}H_{j\nu}^{s,\Psi}*b_{\mathrm d,s-j}
\right\|_2^2
\lesssim_d
\|\Psi\|_\infty^2\alpha\|f\|_1\,2^{-2s(d-1)}.
\end{equation}
Write
\[
u_j=p_{s-j}fp_{s-j},
\qquad
v_j=p_{s-j}f_{s-j}p_{s-j},
\qquad
b_{\mathrm d,s-j}=u_j-v_j.
\]
Both $u_j$ and $v_j$ are positive. We first estimate the family
$(u_j)_j$. Since $I$ is finite, expansion of the squared norm gives
\begin{equation}\label{eq:diag-u-expansion}
\left\|\sum_{j\in I}H_{j\nu}^{s,\Psi}*u_j\right\|_2^2
\leq
2\sum_{j\in I}\sum_{\substack{i\in I\\i\leq j}}
\left|
\left\langle H_{i\nu}^{s,\Psi}*u_i,
H_{j\nu}^{s,\Psi}*u_j\right\rangle_2
\right|.
\end{equation}
By \eqref{eq:convolutionpairing}, positivity, and traciality,
\[
\begin{aligned}
&\left|\left\langle H_{i\nu}^{s,\Psi}*u_i,
H_{j\nu}^{s,\Psi}*u_j\right\rangle_2\right|\\
&\quad\leq
\iint
|(\widetilde H_{j\nu}^{s,\Psi}*H_{i\nu}^{s,\Psi})(x-y)|
\tau\bigl(u_i(y)u_j(x)\bigr)\dd y\dd x.
\end{aligned}
\]
Summing first in $i\leq j$ and applying
Lemma~\ref{lem:sector-packing} with $r=s$, we obtain
\[
\begin{aligned}
\left\|\sum_{j\in I}H_{j\nu}^{s,\Psi}*u_j\right\|_2^2
&\lesssim_d
\alpha\|\Psi\|_\infty^2 2^{-2s(d-1)}
\sum_{j\in I}\varphi(u_j)\\
&=
\alpha\|\Psi\|_\infty^2 2^{-2s(d-1)}
\sum_{j\in I}\varphi(fp_{s-j})\\
&\lesssim_d
\alpha\|\Psi\|_\infty^2 2^{-2s(d-1)}\|f\|_1,
\end{aligned}
\]
where the last step follows from \eqref{eq:massbudget}.

The estimate for $(v_j)_j$ follows in the same way, using the
version of Lemma~\ref{lem:sector-packing} with $f$ replaced by
$f_{s-i}$. Moreover,
\[
\varphi(v_j)
=
\varphi(p_{s-j}f_{s-j}p_{s-j})
=
\varphi(fp_{s-j}),
\]
and hence
\[
\left\|\sum_{j\in I}H_{j\nu}^{s,\Psi}*v_j\right\|_2^2
\lesssim_d
\alpha\|\Psi\|_\infty^2
2^{-2s(d-1)}\|f\|_1.
\]
This proves \eqref{eq:diagsector}. Inserting it into
\eqref{eq:diag-overlap-reduction} and using
\eqref{eq:angularpartition}, we finally get
\[
\begin{aligned}
\left\|
\sum_{j\in I}\Gamma_j^{s,\Psi}*b_{\mathrm d,s-j}
\right\|_2^2
&\lesssim_d
2^{s(d-3/2)}2^{s(d-1)}
\|\Psi\|_\infty^2\alpha\|f\|_1
2^{-2s(d-1)}\\
&=
C_d\|\Psi\|_\infty^2\alpha\|f\|_1\,2^{-s/2}.
\end{aligned}
\]
This is the desired estimate.
\end{proof}

\section{Proofs of the off-diagonal lemmas}\label{sec:offproof}
	
	In this section, we prove Lemmas~\ref{lem:cumulative},
	\ref{lem:martingale}, and~\ref{lem:offprincipal}. They are used in
	Section~\ref{sec:plan}, respectively, for the terms involving
	$K_j^{\Omega_{n+\ell}^{\hi}}
	*(b_{n-j,\ell}+g_{n-j,\ell})$,
	$K_j^{\Omega_{n+\ell}^{\lo}}*g_{n-j,\ell}$, and
	$\Gamma_j^{n+\ell,\Omega_{n+\ell}^{\lo}}*b_{n-j,\ell}$.
	
	\subsection{Proof of Lemma~\ref{lem:cumulative}}
	\label{sec:cumulativeproof}
	
	\begin{proof}

		Fix an integer  $h\geq n_0$ and $j\in\Z$. Since the definitions of
		$b_{n-j,\ell}$ and $g_{n-j,\ell}$ involve the two indices
		$n-j$ and $n+\ell-j$, we reindex the sum by setting
		\[
		k=n-j,\qquad r=n+\ell-j=k+\ell.
		\]
		Then
		\[
		n\geq n_0,\quad \ell\geq1,\quad n+\ell\leq h
		\]
		is equivalent to
		\[
		a:=n_0-j\leq k<r\leq h-j:=b.
		\]
		Thus
		\[
		\sum_{\substack{n\geq n_0,\ \ell\geq1\\ n+\ell\leq h}}
		(b_{n-j,\ell}+g_{n-j,\ell})
		=
		\sum_{a\leq k<r\leq b}
		(b_{k,r-k}+g_{k,r-k}).
		\]
		By the definitions of $b_{k,\ell}$ and $g_{k,\ell}$ in
		\eqref{eq-bkl-gkl}, for $r>k$,
		\[
		\begin{aligned}
			b_{k,r-k}+g_{k,r-k}
			&=
			\bigl[
			p_k(f-f_r)p_r+p_kdf_rq_{r-1}
			\bigr]\\
			&\quad+
			\bigl[
			p_r(f-f_r)p_k+q_{r-1}df_rp_k
			\bigr].
		\end{aligned}
		\]
		Since $f$, $f_r$, and $df_r$ are self-adjoint, the second bracket is
		the adjoint of the first.
		
		We first sum the first bracket. Using
		$q_{r-1}=p_r+q_r$ and $df_r=f_r-f_{r-1}$, we have
		\[
		p_k(f-f_r)p_r+p_kdf_rq_{r-1}
		=
		p_kfp_r+p_kf_rq_r-p_kf_{r-1}q_{r-1}.
		\]
		Therefore, for $a\leq k<b$,
		\[
		\begin{aligned}
			&\sum_{r=k+1}^{b}
			\bigl[
			p_k(f-f_r)p_r+p_kdf_rq_{r-1}
			\bigr]\\
			&=
			\sum_{r=k+1}^{b}p_kfp_r
			+p_kf_bq_b-p_kf_kq_k\\
			&=
			\sum_{r=k+1}^{b}p_kfp_r+p_kf_bq_b,
		\end{aligned}
		\]
		where $p_kf_kq_k=0$.
				Taking adjoints and adding the two identities gives
		\[
		\begin{aligned}
			&\sum_{r=k+1}^{b}
			\bigl(b_{k,r-k}+g_{k,r-k}\bigr)\\
			&=
			\sum_{r=k+1}^{b}
			(p_kfp_r+p_rfp_k)
			+p_kf_bq_b+q_bf_bp_k.
		\end{aligned}
		\]
		Summing over $a\leq k<b$, we obtain
		\[
		\begin{aligned}
			&\sum_{\substack{n\geq n_0,\ \ell\geq1\\ n+\ell\leq h}}
			(b_{n-j,\ell}+g_{n-j,\ell})\\
			&=
			\sum_{a\leq k<r\leq b}
			(p_kfp_r+p_rfp_k)
			+\sum_{k=a}^{b-1}
			\bigl(p_kf_bq_b+q_bf_bp_k\bigr).
		\end{aligned}
		\]
		Set
		$
		P=\sum_{k=a}^{b}p_k.
		$  Since the $p_k$ are pairwise orthogonal,
		\[
		\sum_{a\leq k<r\leq b}(p_kfp_r+p_rfp_k)
		=
		PfP-\sum_{k=a}^{b}p_kfp_k.
		\]
		Moreover, $p_bf_bq_b=q_bf_bp_b=0$, and therefore
		\[
		\sum_{k=a}^{b-1}
		\bigl(p_kf_bq_b+q_bf_bp_k\bigr)
		=
		Pf_bq_b+q_bf_bP.
		\]
		Consequently,
		\begin{equation}\label{eq:cumulativeidentity}
			\begin{split}
				&\sum_{\substack{n\geq n_0,\ \ell\geq1\\ n+\ell\leq h}}
				(b_{n-j,\ell}+g_{n-j,\ell})\\
				&=
				PfP-\sum_{k=a}^{b}p_kfp_k
				+Pf_bq_b+q_bf_bP.
			\end{split}
		\end{equation}	
		We now estimate the four terms on the right-hand side. Set
		\[
		\Lambda=\sum_{k=a}^{b}\varphi(fp_k).
		\]
		We have
		\[
		\|PfP\|_1
		=
		\varphi(PfP)
		=
		\varphi(fP)
		=
		\Lambda,
		\]
		and
		\[
		\left\|
		\sum_{k=a}^{b}p_kfp_k
		\right\|_1
		=
		\sum_{k=a}^{b}\varphi(p_kfp_k)
		=
		\Lambda.
		\]
		It remains to estimate the boundary term $Pf_bq_b$. By
		\eqref{eq:massbudget},
		\[
		\|P\|_2^2
		=
		\varphi(P)
		=
		\sum_{k=a}^{b}\varphi(p_k)
		\leq
		\alpha^{-1}\Lambda.
		\]
		Moreover,
		\[
		\|f_b^{1/2}q_bf_b^{1/2}\|_\infty
		=
		\|q_bf_bq_b\|_\infty
		\leq\alpha,
		\]
		and hence
		\[
		Pf_bq_bf_bP
		=
		Pf_b^{1/2}
		\bigl(f_b^{1/2}q_bf_b^{1/2}\bigr)
		f_b^{1/2}P
		\leq
		\alpha Pf_bP.
		\]
		Therefore
		\[
		\begin{aligned}
			\|Pf_bq_b\|_2^2
			&=
			\varphi(Pf_bq_bf_bP)
			\leq
			\alpha\varphi(Pf_bP)\\
			&=
			\alpha\varphi(f_bP)
			=
			\alpha\varphi(fP)
			=
			\alpha\Lambda.
		\end{aligned}
		\]
		Here  the trace-preserving property of
		$\E_b$ gives $\varphi(f_bP)=\varphi(fP)$.
		H\"older's inequality gives
		\[
		\|Pf_bq_b\|_1
		\leq
		\|P\|_2\|Pf_bq_b\|_2
		\leq
		\Lambda.
		\]
		Its adjoint satisfies
		\[
		\|q_bf_bP\|_1
		=
		\|Pf_bq_b\|_1
		\leq
		\Lambda.
		\]
		Applying the triangle inequality to
		\eqref{eq:cumulativeidentity}, we obtain
		\[
		\left\|
		\sum_{\substack{n\geq n_0,\ \ell\geq1\\n+\ell\leq h}}
		(b_{n-j,\ell}+g_{n-j,\ell})
		\right\|_1
		\leq
		4\Lambda
		=
		4\sum_{k=n_0-j}^{h-j}\varphi(fp_k),
		\]
		which proves \eqref{eq:cumulativebounds}.
		
		Finally, summing in $j$ gives
		\[
		\begin{aligned}
			&\sum_{j\in\Z}
			\left\|
			\sum_{\substack{n\geq n_0,\ \ell\geq1\\n+\ell\leq h}}
			(b_{n-j,\ell}+g_{n-j,\ell})
			\right\|_1\\
			&\quad\leq
			4\sum_{j\in\Z}
			\sum_{k=n_0-j}^{h-j}\varphi(fp_k)\\
			&\quad=
			4\sum_k\varphi(fp_k)
			\#\{j\in\Z:n_0-j\leq k\leq h-j\}\\
			&\quad=
			4(h-n_0+1)\sum_k\varphi(fp_k)\\
			&\quad\leq
			4(h-n_0+1)\norm[1]{f},
		\end{aligned}
		\]
		which is \eqref{eq:cumulativeglobal}.
	\end{proof}

	\subsection{Proof of Lemma~\ref{lem:martingale}}
	\label{sec:martingaleproof}
	The proof has two parts. We first establish an $L_2$ estimate for
	$h_j$ by adapting the argument in the proof of
	\cite[Lemma~3.10]{Lai}.  We then combine this
	estimate with the Fourier decay \eqref{eq:boundedmultiplier}.
	\begin{proof}
		
		For each $t\in\Z$, let $\Delta_t$ be the Fourier multiplier defined by
		\[
		\widehat{\Delta_t h}(\xi)
		=
		\beta(2^{-t}|\xi|)\widehat h(\xi).
		\]
		Denote
		\[
		\beta_t(\xi)=\beta(2^{-t}|\xi|),
		\qquad
		\check\beta_t=\mathcal F^{-1}\beta_t.
		\]
		Then
		\[
		\Delta_t h
		=
		\mathcal F^{-1}(\beta_t\widehat h)
		=
		\check\beta_t*h.
		\]	
	We first establish
	\begin{equation}\label{eq:martcancel}
		\|\Delta_t h_j\|_2
		\lesssim_d
		\min\{1,2^{j+t-s}\}\|h_j\|_2,
		\qquad j\in I,\ t\in\Z.
	\end{equation}
For $j+t\le s-1$, the same argument as in the proof of
\cite[Lemma~3.10]{Lai} gives
\[
\|\Delta_t h_j\|_2
\lesssim_d
2^{j+t-s}\|h_j\|_2.
\]
For $j+t\ge s$, Plancherel's theorem gives
\[
\begin{aligned}
	\|\Delta_t h_j\|_2^2
	&=
	(2\pi)^{-d}
	\int_{\R^d}
	\beta_t(\xi)^2
	\|\widehat h_j(\xi)\|_{L_2(\M)}^2\dd\xi\\
	&\leq
	\|\beta\|_\infty^2\|h_j\|_2^2.
\end{aligned}
\]
Combining the two estimates yields \eqref{eq:martcancel}.

		We next combine \eqref{eq:martcancel} with the Fourier decay of
		$K_j^\Psi$. On the support of $\beta_t$,
		\[
		2^{t-1}<|\xi|<2^{t+1},
		\]
		and therefore \eqref{eq:boundedmultiplier} gives
		\[
		|\widehat K_j^\Psi(\xi)|
		\lesssim_d
		\|\Psi\|_\infty
		\min\{1,2^{-j-t}\}.
		\]
		By Plancherel's theorem,
		\[
		\begin{aligned}
			\|\Delta_t(K_j^\Psi*h_j)\|_2^2
			&=
			(2\pi)^{-d}
			\int_{\R^d}
			\beta_t(\xi)^2
			|\widehat K_j^\Psi(\xi)|^2
			\|\widehat h_j(\xi)\|_{L_2(\M)}^2\dd\xi\\
			&\lesssim_d
			\|\Psi\|_\infty^2
			\min\{1,2^{-j-t}\}^2
			\|\Delta_t h_j\|_2^2.
		\end{aligned}
		\]
		Using \eqref{eq:martcancel}, we obtain
		\begin{equation}\label{eq:martfreqbound}
			\|\Delta_t(K_j^\Psi*h_j)\|_2
			\lesssim_d
			\|\Psi\|_\infty
			w_s(j+t)\|h_j\|_2,
		\end{equation}
		where,  for $u\in\Z$,
		\[
		w_s(u)
		=
		\min\{1,2^{-u}\}
		\min\{1,2^{u-s}\}.
		\]
		Since
		\[
		\sum_{t\in\Z}\beta_t(\xi)^2=1
		\qquad(\xi\ne0),
		\]
		Plancherel's theorem  gives the following identity
		\[
		\left\|
		\sum_{j\in I}K_j^\Psi*h_j
		\right\|_2^2
		=
		\sum_{t\in\Z}
		\left\|
		\Delta_t
		\left(
		\sum_{j\in I}K_j^\Psi*h_j
		\right)
		\right\|_2^2.
		\]
		Since $I$ is finite, the triangle inequality   and
		\eqref{eq:martfreqbound} yield
		\[
		\begin{aligned}
			\left\|
			\sum_{j\in I}K_j^\Psi*h_j
			\right\|_2^2
			&\lesssim_d
			\|\Psi\|_\infty^2
			\sum_{t\in\Z}
			\left[
			\sum_{j\in I}
			w_s(j+t)\|h_j\|_2
			\right]^2.
		\end{aligned}
		\]
		For each fixed $t$, the Cauchy--Schwarz inequality gives
		\[
		\begin{aligned}
			\left[
			\sum_{j\in I}
			w_s(j+t)\|h_j\|_2
			\right]^2
			&\leq
			\left[
			\sum_{j\in I}w_s(j+t)
			\right]
			\left[
			\sum_{j\in I}
			w_s(j+t)\|h_j\|_2^2
			\right]\\
			&\lesssim
			(1+s)2^{-s}
			\sum_{j\in I}
			w_s(j+t)\|h_j\|_2^2,
		\end{aligned}
		\]
		where we used
		$
		\sum_{u\in\Z}w_s(u)
		\lesssim
		(1+s)2^{-s}.
		$
		Therefore,
		\[
		\begin{aligned}
			\sum_{t\in\Z}
			\left[
			\sum_{j\in I}
			w_s(j+t)\|h_j\|_2
			\right]^2
			&\lesssim
			(1+s)2^{-s}
			\sum_{j\in I}
			\|h_j\|_2^2
			\sum_{t\in\Z}w_s(j+t)\\
			&\lesssim
			(1+s)^2 2^{-2s}
			\sum_{j\in I}\|h_j\|_2^2.
		\end{aligned}
		\]
		This gives
		\[
		\left\|
		\sum_{j\in I}K_j^\Psi*h_j
		\right\|_2
		\lesssim_d
		\|\Psi\|_\infty
		(1+s)2^{-s}
		\left(
		\sum_{j\in I}\|h_j\|_2^2
		\right)^{1/2}.
		\]
		We finish the proof.
	\end{proof}

\subsection{Proof of Lemma~\ref{lem:offprincipal}}
\label{sec:twosided}

\begin{proof}
	By \eqref{eq:Seegerkernels} and Lemma~\ref{lem:overlap},
	\begin{equation}\label{eq:off-overlap-reduction}
		\begin{aligned}
			\left\|
			\sum_{j\in I}
			\Gamma_j^{s,\Psi}*b_{n-j,\ell}
			\right\|_2^2
			&=
			\left\|
			\sum_\nu
			P_\nu^s
			\left(
			\sum_{j\in I}
			H_{j\nu}^{s,\Psi}*b_{n-j,\ell}
			\right)
			\right\|_2^2\\
			&\leq
			C_d2^{s(d-3/2)}
			\sum_\nu
			\left\|
			\sum_{j\in I}
			H_{j\nu}^{s,\Psi}*b_{n-j,\ell}
			\right\|_2^2.
		\end{aligned}
	\end{equation}
	In view of \eqref{eq:angularpartition}, to prove
	\eqref{eq:offprincipal}, it suffices to show, uniformly in $\nu$, that
	\begin{equation}\label{eq:off-sector-target}
		\left\|
		\sum_{j\in I}
		H_{j\nu}^{s,\Psi}*b_{n-j,\ell}
		\right\|_2^2
		\lesssim_d
		\norm[\infty]{\Psi}^2
		(1+\ell)\alpha\norm[1]{f}\,
		2^{-2s(d-1)}.
	\end{equation}
	By \eqref{eq-bkl-gkl},
	\[
	\begin{aligned}
		b_{n-j,\ell}
		&=
		p_{n-j}fp_{s-j}
		-p_{n-j}f_{s-j}p_{s-j}
		+
		p_{s-j}fp_{n-j}
		-p_{s-j}f_{s-j}p_{n-j}.
	\end{aligned}
	\]
	The last two terms are adjoints of the first two.
	For a fixed $\nu$, set
	\[
	u_j
	=
	H_{j\nu}^{s,\Psi}*(p_{n-j}fp_{s-j}),
	\qquad
	v_j
	=
	H_{j\nu}^{s,\Psi}*(p_{n-j}f_{s-j}p_{s-j}).
	\]
	It therefore remains to prove
	\begin{equation}\label{eq:off-fixed-two}
		\left\|\sum_{j\in I}u_j\right\|_2^2
		+
		\left\|\sum_{j\in I}v_j\right\|_2^2
		\lesssim_d
		\norm[\infty]{\Psi}^2
		(1+\ell)\alpha\norm[1]{f}\,
		2^{-2s(d-1)}.
	\end{equation}	
	Expanding the squared norm, we have
	\begin{equation}\label{eq-u-v}
		\left\|\sum_{j\in I}u_j\right\|_2^2
		=
		\sum_{j\in I}\|u_j\|_2^2
		+
		2\operatorname{Re}
		\sum_{\substack{i,j\in I\\ i<j}}
		\langle u_i,u_j\rangle_2,
	\end{equation}
	and the same identity holds for $(v_j)_{j\in I}$.
	
	\medskip
	\noindent\textit{Single-scale terms.}
	For $u_j$, the column H\"older inequality \eqref{eq:column22} gives, for every $x\in\R^d$,
	\[
	\begin{aligned}
		\norm[L_2(\M)]{u_j(x)}^2
		&\leq
		\left\|
		\int_{\R^d}
		|H_{j\nu}^{s,\Psi}(x-y)|
		p_{s-j}(y)f(y)p_{s-j}(y)\dd y
		\right\|_\infty\\
		&\qquad\times
		\int_{\R^d}
		|H_{j\nu}^{s,\Psi}(x-y)|
		\tau(p_{n-j}(y)f(y)p_{n-j}(y))\dd y.
	\end{aligned}
	\]
	The $L_\infty$-factor is estimated using \eqref{eq:sectornorms},
	\eqref{eq:finecubecover}, and \eqref{eq:cubeproperties}:
	\[
	\begin{aligned}
		& \left\|\int_{\R^d}|H_{j \nu}^{s, \Psi}(x-y)|
		p_{s-j}(y) f(y) p_{s-j}(y) \dd y\right\|_{\infty} \\
		& \quad \leq
		\|H_{j \nu}^{s, \Psi}\|_{\infty}
		\left\|\sum_{\substack{Q \in \Q_{s-j} \\
				Q \cap\left(x-\operatorname{supp} H_{j \nu}^{s, \Psi}\right) \neq \varnothing}}
		\int_Q p_{s-j}(y) f(y) p_{s-j}(y) \dd y\right\|_{\infty} \\
		& \quad \lesssim_d \alpha\|\Psi\|_{\infty} 2^{-j d}
		\left\|\sum_{\substack{Q \in \Q_{s-j} \\
				Q \cap\left(x-\operatorname{supp} H_{j \nu}^{s, \Psi}\right) \neq \varnothing}}
		|Q| \pi_Q\right\|_{\infty} \\
		& \quad \lesssim_d \alpha\|\Psi\|_{\infty} 2^{-s(d-1)} .
	\end{aligned}
	\]
	Here, for $Q\in\Q_{s-j}$, \eqref{eq:cubeproperties} gives
	\[
	\int_Qp_{s-j}(y)f(y)p_{s-j}(y)\dd y
	\leq
	2^d\alpha|Q|\pi_Q.
	\]
	Integrating in $x$, summing over $j \in I$, and using Fubini's theorem,
	\eqref{eq:sectornorms}, and \eqref{eq:massbudget}, we obtain
	\begin{equation}\label{eq:off-single-raw}
		\begin{aligned}
			\sum_{j \in I}\left\|u_j\right\|_2^2
			& \lesssim_d
			\alpha\|\Psi\|_{\infty} 2^{-s(d-1)}
			\sum_{j \in I}\|H_{j \nu}^{s, \Psi}\|_1
			\varphi\left(p_{n-j} f p_{n-j}\right) \\
			& =
			\alpha\|\Psi\|_{\infty} 2^{-s(d-1)}
			\sum_{j \in I}\|H_{j \nu}^{s, \Psi}\|_1
			\varphi\left(f p_{n-j}\right) \\
			& \lesssim_d
			\alpha\|\Psi\|_{\infty}^2 2^{-2 s(d-1)}
			\sum_{j \in I} \varphi\left(f p_{n-j}\right) \\
			& \lesssim_d
			\alpha\|\Psi\|_{\infty}^2 2^{-2 s(d-1)}\|f\|_1 .
		\end{aligned}
	\end{equation}
	
	For $v_j$, since $n-j \leq s-j$, we have $p_{n-j} \in \A_{s-j}$. Moreover,
	\[
	\begin{aligned}
		\left\|p_{n-j} f_{s-j} p_{s-j}\right\|_2^2
		& \leq
		\left\|p_{s-j} f_{s-j} p_{s-j}\right\|_{\infty}
		\varphi\left(p_{n-j} f_{s-j} p_{n-j}\right) \\
		& \leq 2^d \alpha \varphi\left(f p_{n-j}\right),
	\end{aligned}
	\]
	where we used \eqref{eq:cubeproperties} and the trace-preserving property of $\E_{s-j}$.
	Hence, by Young's inequality \eqref{eq-Young}, \eqref{eq:sectornorms}, and
	\eqref{eq:massbudget},
	\begin{equation}\label{eq:off-single-avg}
		\begin{aligned}
			\sum_{j \in I}\left\|v_j\right\|_2^2
			& \leq
			\sum_{j \in I}\|H_{j \nu}^{s, \Psi}\|_1^2
			\left\|p_{n-j} f_{s-j} p_{s-j}\right\|_2^2 \\
			& \lesssim_d
			\alpha\|\Psi\|_{\infty}^2 2^{-2 s(d-1)}
			\sum_{j \in I} \varphi\left(f p_{n-j}\right)  \\
			& \lesssim_d
			\alpha\|\Psi\|_{\infty}^2 2^{-2 s(d-1)}\|f\|_1 .
		\end{aligned}
	\end{equation}

	\noindent\textit{Cross-scale terms.}
	We next split the cross-scale interactions of \eqref{eq-u-v} into two parts:
	the nearby interactions,
	$
	1\leq j-i\leq \ell,
	$
	and the far interactions,
	$
	i\leq j-\ell-1.
$

	Consider first the range $1\le j-i\le\ell$. Then
	\[
	2\sum_{\substack{i,j\in I\\1\leq j-i\leq\ell}}
	|\langle u_i,u_j\rangle_2|
	\leq
	\sum_{\substack{i,j\in I\\1\leq j-i\leq\ell}}
	\bigl(\|u_i\|_2^2+\|u_j\|_2^2\bigr)
	\leq
	2\ell\sum_{j\in I}\|u_j\|_2^2.
	\]
	Hence, by \eqref{eq:off-single-raw},
	\begin{equation}\label{eq:off-near-raw}
		2\operatorname{Re}\sum_{\substack{i,j\in I\\1\le j-i\le\ell}}
		\langle u_i,u_j\rangle_2
		\leq
		2\sum_{\substack{i,j\in I\\1\le j-i\le\ell}}
		|\langle u_i,u_j\rangle_2|
		\lesssim_d
		\ell\alpha\|\Psi\|_\infty^2
		2^{-2s(d-1)}\|f\|_1.
	\end{equation}
	Similarly, \eqref{eq:off-single-avg} yields
	\begin{equation}\label{eq:off-near-avg}
		2\operatorname{Re}\sum_{\substack{i,j\in I\\1\le j-i\le\ell}}
		\langle v_i,v_j\rangle_2
		\leq
		2\sum_{\substack{i,j\in I\\1\le j-i\le\ell}}
		|\langle v_i,v_j\rangle_2|
		\lesssim_d
		\ell\alpha\|\Psi\|_\infty^2
		2^{-2s(d-1)}\|f\|_1.
	\end{equation}

	Consider now the range $i\le j-\ell-1$.
Since $s=n+\ell$, the far range $i\le j-\ell-1$ is contained
in the summation range of Lemma~\ref{lem:sector-packing} both
for $r=n$ and for $r=s$.  
Lemma~\ref{lem:sector-packing} therefore gives, for
$r\in\{n,s\}$,
\begin{equation}\label{eq:off-far-potential}
	\begin{aligned}
		&\sum_{\substack{i\in I\\ i\le j-\ell-1}}
		\int_{\R^d}
		|(\widetilde H_{j\nu}^{s,\Psi}
		*H_{i\nu}^{s,\Psi})(x-y)|
		p_{r-i}(y)f(y)p_{r-i}(y)\dd y\\
		&\qquad\lesssim_d
		\alpha\|\Psi\|_\infty^2
		2^{-2s(d-1)}\one_\M.
	 	\end{aligned}
\end{equation}
	Multiplying by $p_{r-j}(x)f(x)p_{r-j}(x)$, taking the trace,
	integrating in $x$, and summing over $j\in I$, \eqref{eq:massbudget} gives
	\begin{equation}\label{eq:off-far-positive}
		\begin{aligned}
			&\sum_{\substack{i,j\in I\\i\leq j-\ell-1}}
			\iint_{\R^d\times\R^d}
			|(\widetilde H_{j\nu}^{s,\Psi}
			*H_{i\nu}^{s,\Psi})(x-y)|\\
			&\qquad\times
			\tau\bigl(
			p_{r-i}(y)f(y)p_{r-i}(y)
			p_{r-j}(x)f(x)p_{r-j}(x)
			\bigr)\dd y\dd x\\
			&\qquad\lesssim_d
			\alpha\norm[\infty]{\Psi}^2
			2^{-2s(d-1)}\norm[1]{f},
			\qquad r\in\{n,s\}.
		\end{aligned}
	\end{equation}
	By the convolution identity \eqref{eq:convolutionpairing},
	\[
	\begin{aligned}
		\langle u_i,u_j\rangle_2
		&=
		\iint
		(\widetilde H_{j\nu}^{s,\Psi}*H_{i\nu}^{s,\Psi})(x-y)\\
		&\qquad\times
		\tau\!\left(
		p_{n-i}(y)f(y)p_{s-i}(y)
		p_{s-j}(x)f(x)p_{n-j}(x)
		\right)\dd y\dd x .
	\end{aligned}
	\]
The Cauchy--Schwarz inequality yields
	\[
	\begin{aligned}
		&\big|
		\tau\bigl(
		p_{n-i}(y)f(y)p_{s-i}(y)
		p_{s-j}(x)f(x)p_{n-j}(x)
		\bigr)
		\big|\\
		&\quad\leq
		\tau\bigl(
		p_{n-i}(y)f(y)p_{n-i}(y)
		p_{n-j}(x)f(x)p_{n-j}(x)
		\bigr)^{1/2}\\
		&\qquad\times
		\tau\bigl(
		p_{s-i}(y)f(y)p_{s-i}(y)
		p_{s-j}(x)f(x)p_{s-j}(x)
		\bigr)^{1/2}.
	\end{aligned}
	\]
	Therefore, the Cauchy--Schwarz inequality again in the sum over $(i,j)$ and the integral over $(x,y)$,
	together with \eqref{eq:off-far-positive} for $r=n$ and $r=s$, yields
	\[
	\begin{aligned}
	 \sum_{\substack{i,j\in I\\i\leq j-\ell-1}}
		\big|\langle u_i,u_j\rangle_2\big| 
		& \leq
		\Biggl[
		\sum_{\substack{i,j\in I\\i\leq j-\ell-1}}
		\iint
		|(\widetilde H_{j\nu}^{s,\Psi}*H_{i\nu}^{s,\Psi})(x-y)|\\
		& \times
		\tau\bigl(
		p_{n-i}(y)f(y)p_{n-i}(y)
		p_{n-j}(x)f(x)p_{n-j}(x)
		\bigr)\dd y\dd x
		\Biggr]^{1/2}\\
		& \times
		\Biggl[
		\sum_{\substack{i,j\in I\\i\leq j-\ell-1}}
		\iint
		|(\widetilde H_{j\nu}^{s,\Psi}*H_{i\nu}^{s,\Psi})(x-y)|\\
		&\hspace{8mm}\times
		\tau\bigl(
		p_{s-i}(y)f(y)p_{s-i}(y)
		p_{s-j}(x)f(x)p_{s-j}(x)
		\bigr)\dd y\dd x
		\Biggr]^{1/2}\\
		&\lesssim_d
		\alpha\norm[\infty]{\Psi}^2
		2^{-2s(d-1)}\norm[1]{f}.
	\end{aligned}
	\]
	Consequently,
	\begin{equation}\label{eq:off-far-raw}
		\begin{aligned}
			2\operatorname{Re}
			\sum_{\substack{i,j\in I\\ i\le j-\ell-1}}
			\langle u_i,u_j\rangle_2
			&\leq
			2\sum_{\substack{i,j\in I\\ i\le j-\ell-1}}
			|\langle u_i,u_j\rangle_2|\\
			&\lesssim_d
			\alpha\|\Psi\|_\infty^2
			2^{-2s(d-1)}\|f\|_1.
		\end{aligned}
	\end{equation}
	
	The same argument applies to the cross terms associated with $v_j$.
	Indeed, the second assertion of Lemma~\ref{lem:sector-packing} applies with
	$f_{s-i}$ in place of $f$, and
	\[
	\varphi(p_{r-j}f_{s-j}p_{r-j})
	=
	\varphi(fp_{r-j}),
	\qquad r\in\{n,s\},
	\]
	by the trace-preserving property of $\E_{s-j}$. Consequently,
	\begin{equation}\label{eq:off-far-avg}
		2\operatorname{Re}
		\sum_{\substack{i,j\in I\\ i\le j-\ell-1}}
		\langle v_i,v_j\rangle_2
		\lesssim_d
		\alpha\|\Psi\|_\infty^2
		2^{-2s(d-1)}\|f\|_1.
	\end{equation}
	
	Combining \eqref{eq:off-single-raw}--\eqref{eq:off-near-avg},
	\eqref{eq:off-far-raw}, and \eqref{eq:off-far-avg} yields
	\eqref{eq:off-fixed-two}. Since the remaining two terms in
	$b_{n-j,\ell}$ are adjoints, \eqref{eq:off-sector-target} follows.
	Finally, \eqref{eq:off-overlap-reduction}, \eqref{eq:off-sector-target},
	and \eqref{eq:angularpartition} give \eqref{eq:offprincipal}.
\end{proof}


\section{Nonhomogeneous kernels}\label{sec:nonhomogeneous}

Let $K\in\mathscr S'(\R^d)$ be locally integrable
away from the origin, with
$\widehat K\in L_\infty(\R^d)$.
Let $T_Kf=K*f$ denote the associated convolution
operator on $L_2(\A)$, defined via the Fourier
multiplier $\widehat K$.
 Following
\cite[(1.4)--(1.7)]{Seeger}, set
\[
V_R(\theta)=\int_R^{2R}|K(r\theta)|r^{d-1}\dd r,
\qquad V(\theta)=\sup_{R>0}V_R(\theta),
\]
\[
\eta(a)=\sup_{\substack{h>0\\R\geq ah}}
\int_{S^{d-1}}\int_R^{2R}
|K((r-h)\theta)-K(r\theta)|r^{d-1}\dd r\dd\sigma(\theta),
\qquad a\geq2.
\]
Write $C_{V,1}$ for \eqref{eq:angularconstant} with $V$ in place of
$\Omega$, and put
\[
A_K=\|\widehat K\|_\infty+C_{V,1}
+\int_2^\infty\eta(a)\frac{\dd a}{a}.
\]

\begin{theorem}\label{thm:nonhomogeneous}
	If $V\in L(\log^+L)(S^{d-1})$ and
	$\int_2^\infty\eta(a)\dd a/a<\infty$, then $T_K$ extends uniquely from
	$\core$ to a bounded linear map $L_1(\A)\to L_{1,\infty}(\A)$, with
	\begin{equation}\label{eq:nh-main}
		\|T_Kf\|_{1,\infty}\leq C_d A_K\|f\|_1.
	\end{equation}
	
\end{theorem}

\begin{proof}
	Suppose first that $f$ satisfies \eqref{eq:finitemodel}. The case
	$\|V\|_1=0$ is immediate: then $K$ is supported at the origin, and
	$\widehat K\in L_\infty$ implies that $T_K$ is a scalar multiple of
	the identity. Otherwise fix $\lambda>0$ and apply
	Theorem~\ref{thm:regularcz} at height $\alpha=\lambda/A_K$:
	\[
	f=g+ \bd+\boff.
	\]
	By Plancherel's theorem and \(\|g\|_2^2\lesssim_d
	\alpha\|f\|_1\),
	\[
	\|T_Kg\|_2^2
	\lesssim_d
	\alpha\|\widehat K\|_\infty^2\|f\|_1.
	\]
	Let $\zeta$ be the localization projection from
	Lemma~\ref{lem:cz-localization}, and write $b=\bd+\boff$.
	We first estimate $\zeta\sum_{j\in I}K_j*b\,\zeta$
	uniformly over finite $I\subset\mathbb Z$.
	
	Let
	\[
	K_j(x)=\beta(2^{-j}|x|)^2K(x).
	\]
	By Lemma~\ref{lemm:localsupport}, we have
	\[
	\zeta(K_j*b_{\mathrm d,k})\zeta
	=
	\zeta(K_j*b_{\mathrm{off},k})\zeta=0,
	\qquad j+k<n_0.
	\]
	Hence, using \eqref{eq:lagresolution} and setting
	$n=j+k$, we obtain
	\begin{equation}\label{eq:nh-local}
		\zeta\sum_{j\in I}K_j*b\,\zeta
		=\zeta\sum_{j\in I}K_j*
		\Bigg[\sum_{s\ge n_0}b_{\mathrm d,s-j}
		+\sum_{n\ge n_0}\sum_{\ell\ge1}
		(b_{n-j,\ell}+g_{n-j,\ell})\Bigg]\zeta.
	\end{equation}
	We keep the blocks $b_{k,\ell}$ and $g_{k,\ell}$
	together in the following estimates.
	Let $t=2^mn_0$, $m=0,1,\ldots$, and define
	\[
	B_j^t
	=
	\sum_{t\le s<2t}b_{\mathrm d,s-j}
	+
	\sum_{\substack{n\ge n_0,\ \ell\ge1\\
			t\le n+\ell<2t}}
	(b_{n-j,\ell}+g_{n-j,\ell}).
	\]
	Consequently,
	\begin{equation}\label{eq:nh-grouping}
		\zeta\sum_{j\in I}K_j*b\,\zeta
		=\zeta\sum_t\sum_{j\in I}K_j*B_j^t\,\zeta.
	\end{equation}
	Moreover, by Theorem~\ref{thm:regularcz}(ii) and
	\eqref{eq:cumulativeglobal}, we have
	\begin{equation}\label{eq:nh-interval}
		\sum_j\|B_j^t\|_1
		\lesssim_d t\|f\|_1.
	\end{equation}
	To estimate the sum on the
	right-hand side of \eqref{eq:nh-grouping},
	we decompose the kernels
	$K_j$ as follows.
	
	Set $\delta_0=[100(d+2)]^{-1}$ and choose
	$0\le\phi\in C_c^\infty((0,1/16))$ with $\int\phi=1$.
	Following \cite[Section~3]{Seeger}, put
	\[
	J_j^t(r\theta)=\beta(2^{-j}r)^2
	\int_0^\infty K((r-\rho)\theta)
	2^{\delta_0t-j}\phi(2^{\delta_0t-j}\rho)\dd\rho,
	\]
	\[
	\begin{aligned}
		H_j^t(r\theta)
		&=\one_{\{V\le2^{\delta_0t}\|V\|_1\}}(\theta)J_j^t(r\theta),\\
		R_j^t&=J_j^t-H_j^t,\qquad S_j^t=K_j-J_j^t.
	\end{aligned}
	\]
	Thus $K_j=H_j^t+R_j^t+S_j^t$. Seeger's estimates
	\cite[(3.4)--(3.5), (3.9)--(3.10)]{Seeger} give
	\begin{align}
		\sup_{j,r,\theta}r^{d+a}|\partial_r^aH_j^t(r\theta)|
		&\lesssim_{d,a}\|V\|_1\,2^{\delta_0t(a+2)},\quad a\ge0,
		\label{eq:nh-derivatives}\\
		\sup_j\|R_j^t\|_1
		&\lesssim_d\int_{\{V>2^{\delta_0t}\|V\|_1\}}V\dd\sigma,
		\label{eq:nh-angularerror}\\
		\sup_j\|S_j^t\|_1
		&\lesssim_d\eta(2^{\delta_0t+2}).
		\label{eq:nh-radialerror}
	\end{align}
	For $t\le s<2t$, define, as in \eqref{eq:Seegerkernels},
	\[
	H_{j\nu}^s(x)=H_j^t(x)\one_{E_\nu^s}\left(\frac{x}{|x|}\right),
	\qquad
	\Gamma_j^s(x)=\sum_\nu(P_\nu^sH_{j\nu}^s)(x).
	\]
	We use $H_j^t=\Gamma_j^s+(H_j^t-\Gamma_j^s)$ on the diagonal
	blocks and on $b_{n-j,\ell}$, with $s=n+\ell$ in the latter case;
	the full $H_j^t$ is used on $g_{n-j,\ell}$.
	Consequently, \eqref{eq:nh-grouping} becomes
	\begin{equation}\label{eq:nh-two-sums}
		\zeta\sum_{j\in I}K_j*b\,\zeta
		=\zeta(\mathcal U_I+\mathcal E_I)\zeta,
	\end{equation}
	where
	\begin{equation}\label{eq:nh-L2sum}
		\begin{aligned}
			\mathcal U_I=\sum_t\sum_{j\in I}\Bigg[
			\sum_{t\le s<2t}\Gamma_j^s*b_{\mathrm d,s-j}
			+\sum_{\substack{n\ge n_0,\ \ell\ge1\\t\le n+\ell<2t}}
			\bigl(\Gamma_j^{n+\ell}*b_{n-j,\ell}
			+H_j^t*g_{n-j,\ell}\bigr)\Bigg]
		\end{aligned}
	\end{equation}
	and
	\begin{equation}\label{eq:nh-L1sum}
		\begin{aligned}
			\mathcal E_I=\sum_t\sum_{j\in I}\Bigg[
			\sum_{t\le s<2t}(H_j^t-\Gamma_j^s)*b_{\mathrm d,s-j}
			+\sum_{\substack{n\ge n_0,\ \ell\ge1\\t\le n+\ell<2t}}
			(H_j^t-\Gamma_j^{n+\ell})*b_{n-j,\ell}
			+(R_j^t+S_j^t)*B_j^t\Bigg].
		\end{aligned}
	\end{equation}
	
	We first estimate $\mathcal U_I$ in $L_2(\A)$.
	By \eqref{eq:nh-derivatives} with $a=0$, the sector bounds
	\eqref{eq:sectornorms}, the interaction estimate
	\eqref{eq:interactionmajorant}, and the geometric estimates
	\eqref{eq:rectangle}--\eqref{eq:finecubecover} hold for
	$H_{j\nu}^s$, with $\|\Psi\|_\infty$ replaced by
	$C_d\|V\|_1 2^{2\delta_0t}$.
	Thus the proofs of Lemmas~\ref{lem:diagprincipal}
	and~\ref{lem:offprincipal} give
	\[
	\begin{aligned}
		\left\|\sum_{j\in I}
		\Gamma_j^s*b_{\mathrm d,s-j}\right\|_2
		&\lesssim_d
		\|V\|_1 2^{2\delta_0t}2^{-s/4}
		(\alpha\|f\|_1)^{1/2},\\
		\left\|\sum_{j\in I}
		\Gamma_j^s*b_{n-j,\ell}\right\|_2
		&\lesssim_d
		\|V\|_1 2^{2\delta_0t}
		(1+\ell)^{1/2}2^{-s/4}
		(\alpha\|f\|_1)^{1/2},
	\end{aligned}
	\]
	where $s=n+\ell$ in the second estimate.
	By \eqref{eq:nh-derivatives} and radial integration
	by parts, the Fourier decay estimate used in the
	proof of Lemma~\ref{lem:martingale} remains valid
	for $H_j^t$, with $\|\Psi\|_\infty$ replaced by
	$C_d\|V\|_1 2^{4\delta_0t}$.
	Thus, together with Lemma~\ref{lem:lag-blocks}(ii),
	we obtain
	\[
	\left\|\sum_{j\in I}
	H_j^t*g_{n-j,\ell}\right\|_2
	\lesssim_d
	\|V\|_1 2^{4\delta_0t}
	(1+s)2^{-s}
	(\alpha\|f\|_1)^{1/2},
	\qquad s=n+\ell.
	\]
	
	Summing these three estimates as in Section~\ref{sec:plan} gives
	\begin{equation}\label{eq:nh-L2bound}
		\|\mathcal U_I\|_2
		\lesssim_d
		\|V\|_1(\alpha\|f\|_1)^{1/2}.
	\end{equation}
	
	We next estimate $\mathcal E_I$ in $L_1(\A)$.
	For the first two terms in \eqref{eq:nh-L1sum},
	\cite[Lemma~2.2]{Seeger} (used already in the proof of
	Lemma~\ref{lem:remainder}) and \eqref{eq:nh-derivatives} give
	\[
	\|(H_j^t-\Gamma_j^s)*h_Q\|_1
	\lesssim_d\|V\|_1 2^{-s/8}\|h_Q\|_1
	\]
	whenever $\int_Qh_Q=0$ and $Q$ has side length $2^{j-s}$.
	As in Section~\ref{sec:plan}, summing over the bad cubes and
	using Theorem~\ref{thm:regularcz}(ii) and
	Lemma~\ref{lem:lag-blocks}(i) bounds the first two terms
	of \eqref{eq:nh-L1sum} in $L_1$ by $C_d\|V\|_1\|f\|_1$.
	The remaining terms in \eqref{eq:nh-L1sum} are handled by
	\eqref{eq:nh-interval}, \eqref{eq:nh-angularerror}, and
	\eqref{eq:nh-radialerror}:
	\[
	\begin{aligned}
		\Big\|\sum_t\sum_{j\in I}R_j^t*B_j^t\Big\|_1
		&\lesssim_d\|f\|_1\sum_t t
		\int_{\{V>2^{\delta_0t}\|V\|_1\}}V\dd\sigma
		\lesssim_d C_{V,1}\|f\|_1,\\
		\Big\|\sum_t\sum_{j\in I}S_j^t*B_j^t\Big\|_1
		&\lesssim_d\|f\|_1\sum_t t\eta(2^{\delta_0t+2}).
	\end{aligned}
	\]
	Since $\eta$ is nonincreasing and the intervals $[t/2,t]$
	are disjoint,
	\[
	\sum_t t\eta(2^{\delta_0t+2})
	\lesssim_d\int_2^\infty\eta(a)\frac{\dd a}{a}.
	\]
	Combining these estimates gives
	\begin{equation}\label{eq:nh-L1bound}
		\|\mathcal E_I\|_1\lesssim_d A_K\|f\|_1.
	\end{equation}
		By \eqref{eq:nh-two-sums}, \eqref{eq:nh-L2bound}, and
	\eqref{eq:nh-L1bound}, Chebyshev's inequality gives, uniformly in
	finite $I\subset\mathbb Z$,
	\begin{equation}\label{eq:nh-finitebad}
		d_\lambda\left(\zeta\sum_{j\in I}K_j*b\,\zeta\right)
		\lesssim_d
		\left(\lambda^{-2}\alpha\|V\|_1^2+\lambda^{-1}A_K\right)\|f\|_1.
	\end{equation}
Applying the preceding estimate to
$\zeta\sum_{j=-J}^{J}K_j*b\,\zeta$
and letting $J\to\infty$, we obtain the
same bound for $\zeta T_Kb\zeta$.

	By \eqref{eq:bad-reduction}, \eqref{eq:localized-projection},
	and the $L_2$ bound for $T_Kg$, we obtain
	\[
	\begin{aligned}
		d_{4\lambda}(T_Kf)
		&\le d_\lambda(T_Kg)+d_{3\lambda}(T_Kb)\\
		&\le d_\lambda(T_Kg)+d_\lambda(\zeta T_Kb\zeta)
		+2\varphi(1-\zeta)\\
		&\lesssim_d
		\left[\lambda^{-2}\alpha
		(\|\widehat K\|_\infty^2+\|V\|_1^2)
		+\lambda^{-1}A_K+\alpha^{-1}\right]\|f\|_1\\
		&\lesssim_d \frac{A_K}{\lambda}\|f\|_1,
	\end{aligned}
	\]
	where $\alpha=\lambda/A_K$ and $\|V\|_1\le A_K$.
	The extension to $L_1(\A)$ follows as in
	Section~\ref{sec:completion}.
	
\end{proof}

\section*{Acknowledgments}
This paper is supported by National Key R\&D Program of China (No. 2023YFA1010800),  NSFC (Nos. 12125109, W2411005, 12361131578),  Natural Science Foundation of Hunan Province (Nos. 2024JJ1010, 2025ZYJ002, 2024RC1016), China Scholarship Council (No. 202306370152) and Hunan Provincial Innovation Foundation for postgraduate (No. CX20240162).

During the preparation of this work, the authors used GPT-5.6 Sol to assist with language editing and
refinement of mathematical arguments. The authors take full responsibility for the content of this paper.

\end{document}